\documentclass[reqno]{amsart}

\usepackage[utf8]{inputenx} 
\usepackage[T1]{fontenc}    
\usepackage[all]{xy}
\usepackage[margin=4cm]{geometry}
\usepackage[numbers,sort]{natbib}

\usepackage{enumitem}
\setenumerate{label=(\roman*),itemsep=3pt,topsep=3pt}

\usepackage{amssymb}   
\usepackage{amsthm}    
\usepackage{thmtools}  
\usepackage{mathtools} 
\usepackage{mathrsfs}  
\usepackage{tikz}      
\usetikzlibrary{arrows}

\usepackage{graphicx}  
\usepackage{csquotes}  
\usepackage{textcomp}  
\usepackage{listings}  
\usepackage{todonotes}

\usepackage{varioref}
\usepackage{hyperref}
\usepackage[nameinlink, capitalize, noabbrev]{cleveref}

\declaretheorem[style = plain, numberwithin = section]{theorem}
\newtheorem{thmx}{Theorem}

\declaretheorem[style = plain,      sibling = theorem]{lemma}
\declaretheorem[style = plain,      sibling = theorem]{proposition}
\declaretheorem[style = definition, sibling = theorem]{definition}
\declaretheorem[style = definition, sibling = theorem]{example}
\declaretheorem[style = remark,    sibling = theorem]{remark}

\declaretheorem[style = plain,		sibling = theorem]{convention}

\DeclareMathOperator{\spn}{span}

\newcommand{\N}{\mathbb{N}}   
\newcommand{\Z}{\mathbb{Z}}   
\newcommand{\R}{\mathbb{R}}   
\newcommand{\C}{\mathbb{C}}   
\newcommand{\T}{\mathbb{T}}   

\DeclareMathOperator{\tr}{tr}
\newcommand{\E}{\mathcal{E}}  
\newcommand{\F}{\mathcal{F}}  
\newcommand{\Adj}{\mathcal{L}}  

\def\hs#1#2{\left\langle #1,#2\right\rangle} 
\def\lhs#1#2{{_\bullet\!\!}\left\langle #1,#2\right\rangle} 
\def\rhs#1#2{\left\langle #1,#2\right\rangle\!\!{_\bullet}}	
\def\modu1{\mathcal{E}} 
\def\modu2{\mathcal{F}} 
\def\an{C} 
\def\sy{D} 
\def\ft{S} 
\def\modan{\Phi}
\def\modsy{\Psi}
\def\modft{\Theta}

\def\tfp#1{#1 \times \widehat{#1}} 
\def\Sub{\Delta} 

\def\heis#1#2{\mathcal{E}_{#2}(#1)} 
\def\bes#1#2{B_{#2}(#1)}            

\newcommand{\vertiii}[1]{{\left\vert\kern-0.25ex\left\vert\kern-0.25ex\left\vert #1 
    \right\vert\kern-0.25ex\right\vert\kern-0.25ex\right\vert}}

\usepackage{commath}

 \usepackage[foot]{amsaddr}

\author{Are Austad}
\address[Are Austad]{Norwegian University of Science and Technology, Department of Mathematical Sciences, Trondheim, Norway.}
\email{are.austad@ntnu.no, ubenstad@math.uio.no}
\author{Ulrik Enstad}
\address[Ulrik Enstad]{University of Oslo, Department of Mathematics, Oslo, Norway.}

\title{Heisenberg modules as function spaces}
\date{}

\begin{document}

\maketitle

\begin{abstract}
    \noindent
    Let $\Sub$ be a closed, cocompact subgroup of $\tfp{G}$, where $G$ is a second countable, locally compact abelian group. Using localization of Hilbert $C^*$-modules, we show that the Heisenberg module $\heis{G}{\Sub}$ over the twisted group $C^*$-algebra $C^*(\Sub,c)$ due to Rieffel can be continuously and densely embedded into the Hilbert space $L^2(G)$. This allows us to characterize a finite set of generators for $\heis{G}{\Delta}$ as exactly the generators of multi-window (continuous) Gabor frames over $\Sub$, a result which was previously known only for a dense subspace of $\heis{G}{\Delta}$. We show that $\heis{G}{\Sub}$ as a function space satisfies two properties that make it eligible for time-frequency analysis: Its elements satisfy the fundamental identity of Gabor analysis if $\Sub$ is a lattice, and their associated frame operators corresponding to $\Sub$ are bounded.
    
\end{abstract}

\section{Introduction}

Gabor analysis concerns sets of time-frequency shifts of functions. The field has its roots in a paper by the electrical engineer and physicist Dennis Gabor \cite{Ga46}. In this paper, the author made the claim that one could obtain basis-like representations of functions in $L^2(\R)$ in terms of the set $\{ e^{2\pi i lx} \phi(x- k) : k,l \in \Z \}$, where $\phi$ denotes a Gaussian. Today, one of the central problems of the field remains understanding the spanning and basis-like properties of sets of the form $\{ e^{2\pi i \beta l x} \eta(x - \alpha k) : k,l \in \Z \}$ for a given $\eta \in L^2(\R)$ and $\alpha,\beta > 0$.

Although Gabor analysis is usually carried out for functions of one or several real variables, the nature of time-frequency shifts makes it possible to generalize many aspects of the theory to the setting of a locally compact abelian group $G$ \cite{gr98}. In this setting, elements of $G$ represent time, while elements of the Pontryagin dual $\widehat{G}$ represent frequency. If $\eta \in L^2(G)$, then a time-frequency shift of $\eta$ is a function of the form $\pi(x,\omega)\eta(t) = \omega(t) \eta(t-x)$ for $t,x \in G$ and $\omega \in \widehat{G}$. A Gabor system with generator $\eta$ will in general be any collection of time-frequency shifts of $\eta$. In this paper, we will allow continuous Gabor systems over any closed subgroup $\Sub$ of the time-frequency plane $\tfp{G}$, which will be of the form $( \pi(z) \eta)_{z \in \Sub}$. We say that such a system forms a Gabor frame if it is a continuous frame for $L^2(G)$, which means that there exist $C,D > 0$ such that
\[ C \| \xi \|_2^2 \leq \int_{\Sub} | \langle \xi, \pi(z) \eta \rangle |^2 \, \dif z \leq D \| \xi \|_2^2 \]
for every $\xi \in L^2(G)$. Here, we integrate with respect to a fixed Haar measure on $\Sub$. More generally, if $\eta_1, \ldots, \eta_k \in L^2(G)$, one calls $( \pi(z) \eta_j)_{z \in \Sub, 1 \leq j \leq k}$ a multi-window Gabor frame if there exist $C,D > 0$ such that
\[ C \| \xi \|_2^2 \leq \sum_{j=1}^k \int_{\Sub} | \langle \xi, \pi(z) \eta_j \rangle |^2 \, \dif z \leq D \| \xi \|_2^2 \]
for all $\xi \in L^2(G)$. If $\Sub$ is a discrete subgroup of $\tfp{G}$, one recovers the usual notion of a (discrete) regular Gabor frame. Here, regular means that the discrete subset $\Sub$ of $\tfp{G}$ has the structure of a subgroup. A basic fact of Gabor frame theory is that $(\pi(z) \eta)_{z \in \Sub}$ is a Gabor frame if and only if the associated frame operator $\ft_{\eta} \colon L^2(G) \to L^2(G)$ is invertible. The operator is given weakly by
\[ \ft_{\eta} \xi = \int_{\Sub} \langle \xi, \pi(z) \eta \rangle \pi(z) \eta \, \dif z \]
for $\xi \in L^2(G)$.\\

In \cite{Lu09,Lu11,JaLu18}, Luef and later Jakobsen and Luef discovered that the duality theory of regular Gabor frames is closely related to a class of imprimitivity bimodules constructed by Rieffel \cite{Ri88}. These imprimitivity bimodules are known as \emph{Heisenberg modules}. In general, a Hilbert $C^*$-module over a $C^*$-algebra $A$ can be thought of as a generalized Hilbert space where the field of scalars $\C$ is replaced with $A$, and where the inner product takes values in $A$ rather than $\C$. Hilbert $C^*$-modules were introduced by Kaplansky in \cite{Ka53}, and have since become essential in many parts of operator algebras and noncommutative geometry \cite{Co94}. An imprimitivity $A$-$B$-bimodule is both a left Hilbert $C^*$-module over $A$ and a right Hilbert $C^*$-module over $B$, with compatibility conditions on the left and right structures. If there exists an imprimitivity $A$-$B$-bimodule, then the $C^*$-algebras $A$ and $B$ are called Morita equivalent, a notion first described by Rieffel in \cite{Ri74,Ri74b}. Morita equivalent $C^*$-algebras share many important properties, such as representation theory and ideal structure.

For a closed subgroup $\Sub$ of $\tfp{G}$, the Heisenberg module $\heis{G}{\Sub}$ can be constructed as a norm completion of the Feichtinger algebra $S_0(G)$ \cite{Lu09}. The latter is an important space of functions in time-frequency analysis \cite{Fe81}. The Heisenberg module implements the Morita equivalence between the twisted group $C^*$-algebras $C^*(\Sub,c)$ and $C^*(\Sub^{\circ},\overline{c})$. Here, $\Sub^{\circ}$ denotes the adjoint subgroup of $\Sub$, which consists of all points $w \in \tfp{G}$ for which $\pi(w)$ commutes with $\pi(z)$ for every $z \in \Sub$. Readers familiar with Gabor analysis know that the adjoint subgroup plays a central role in results such as the fundamental identity of Gabor analysis, and this result can indeed be inferred directly from the structure of the Heisenberg modules. An important class of examples come from when $G = \R^n$ and $\Sub$ is a lattice in $\tfp{G} \cong \R^{2n}$, in which case the twisted group $C^*$-algebras $C^*(\Sub,c)$ and $C^*(\Sub^{\circ},\overline{c})$ are both noncommutative $2n$-tori. Indeed, these examples were the original motivation for the construction of Heisenberg modules in \cite{Ri88}. However, the construction has also been applied in other contexts, such as in the construction of finitely generated projective modules over noncommutative solenoids \cite{LaPa13,LaPa15,EnJaLu19}.

For a general left Hilbert $C^*$-module $\E$ over a $C^*$-algebra $A$, one defines rank-one operators in analogy with the Hilbert space case. Specifically, if $\eta,\gamma \in \E$, the rank-one operator $\modft_{\eta,\gamma} \colon \E \to \E$ is given by
\[ \modft_{\eta,\gamma} \xi = \lhs{\xi}{\eta} \gamma \]
for $\xi \in \E$. Here, $\lhs{\cdot}{\cdot}$ denotes the $A$-valued inner product on $\E$. A central observation in \cite{Lu09} is that for $\eta \in S_0(G)$, the rank-one operator $\modft_{\eta,\eta}$ associated to the Heisenberg module $\heis{G}{\Sub}$ agrees with the Gabor frame operator $\ft_{\eta}$ on a dense subspace of $\heis{G}{\Sub}$, namely the Feichtinger algebra $S_0(G)$. This observation has an important consequence: It allows a finite generating set of the Heisenberg module coming from the dense subspace $S_0(G)$ to be characterized exactly as the generators of a multi-window Gabor frame over $\Sub$ \cite[p. 14]{JaLu18}. Moreover, such a finite generating set exists (that is, $\heis{G}{\Sub}$ is finitely generated) if and only if $\Sub$ is cocompact in $\tfp{G}$ \cite[Theorem 3.9]{JaLu18}. However, since $\heis{G}{\Sub}$ is an abstract completion of $S_0(G)$, its elements can a priori not be interpreted as functions in any sense. Therefore, it is not straightforward to obtain a similar characterization for generators of $\heis{G}{\Sub}$ not necessarily in $S_0(G)$.

Nonetheless, it was recently remarked in \cite{AuLu18} that $\heis{G}{\Sub}$ can be continuously embedded into $L^2(G)$. In the present paper, we elaborate on this embedding, and show how it arises naturally from the notion of localization of Hilbert $C^*$-modules as discussed in \cite{La95}. The important extra structure on the Heisenberg module when localizing is a faithful trace on the $C^*$-algebra $C^*(\Sub,c)$. In the case that $\Sub$ is a lattice in $\tfp{G}$, we use the canonical tracial state on $C^*(\Sub,c)$ (see e.g.\ \cite[p.\ 951]{BeOm18}). If $\Sub$ is only cocompact, we have to work a bit more, see \cref{prop:bimodule_localization}. It was already observed in \cite{Lu09} that this trace plays an important role when connecting Heisenberg modules and Gabor frames. However, the consequence that the trace makes it possible to embed $\heis{G}{\Sub}$ continuously into $L^2(G)$ was first observed in \cite{AuLu18}.

Furthermore, in the language of localization, the rank-one operator $\modft_{\eta,\eta}$ for $\eta \in \heis{G}{\Sub}$ extends uniquely to a bounded linear operator on $L^2(G)$, and we show in this paper that the extension is exactly the Gabor frame operator $\ft_{\eta}$ (\cref{thm:loc_operators}). As a consequence, we generalize the equivalence between generators of Heisenberg modules and generators of multi-window Gabor frames to the case when the generators 
belong to $\heis{G}{\Sub}$ (\cref{thm:generators_multiwindow}). We summarize some of our main results in the following.

\begin{thmx}[cf.\ \cref{Prop:heis-completion-bessel}, \cref{thm:loc_operators}, \cref{thm:generators_multiwindow}]\label{thm:intro1}
Let $G$ be a second countable, locally compact abelian group, and let $\Sub$ be a closed, cocompact subgroup of $\tfp{G}$. Denote by $B_{\Sub}(G)$ the subspace of $L^2(G)$ consisting of those $\eta \in L^2(G)$ for which $( \pi(z) \eta)_{z \in \Sub}$ is a Bessel family for $L^2(G)$, that is,
\[ \int_{\Sub} | \langle \xi, \pi(z) \eta \rangle |^2 \, \dif z < \infty \]
for every $\xi \in L^2(G)$. This is a Banach space with respect to the norm
\[ \| \eta \|_{B_{\Sub}(G)} = \| S_{\eta} \|^{1/2} = \sup_{ \| \xi \|_2 = 1} \left( \int_{\Sub} | \langle \xi, \pi(z) \eta \rangle |^2 \, \dif z \right)^{1/2}. \]
The following hold:
\begin{enumerate}
    \item The Heisenberg module $\heis{G}{\Sub}$ has a concrete description as the completion of $S_0(G)$ in the Banach space $B_{\Sub}(G)$. The actions are given in \cref{Prop:heis-completion-bessel}.
    \item For $\eta \in \heis{G}{\Sub}$, the Heisenberg module rank-one operator $\modft_{\eta} \colon \heis{G}{\Sub} \to \heis{G}{\Sub}$ extends to the Gabor frame operator $\ft_{\eta} \colon L^2(G) \to L^2(G)$.
    \item Let $\eta_1, \ldots, \eta_k \in \heis{G}{\Sub}$. Then $\{ \eta_1, \ldots, \eta_k \}$ is a generating set for $\heis{G}{\Sub}$ as a left $C^*(\Sub,c)$-module if and only if $( \pi(z) \eta_j)_{z \in \Sub, 1 \leq j \leq k}$ is a multi-window Gabor frame for $L^2(G)$.
\end{enumerate}
\end{thmx}

Part (iii) of \cref{thm:intro1} gives a complete description of finite generating sets of the Heisenberg modules due to Rieffel, showing that they are the generators of a multi-window Gabor frame. Conversely, multi-window Gabor frames over $\Sub$ with generators in $\heis{G}{\Sub}$ give rise to finite generating sets for $\heis{G}{\Sub}$.

Note also that part (i) of \cref{thm:intro1} implies that $( \pi(z) \eta)_{z \in \Sub}$ is a Bessel family for $L^2(G)$ whenever $\eta \in \heis{G}{\Sub}$.
Consequently, the Gabor analysis operator $\an_\eta \colon L^2 (G) \to L^2 (\Sub)$, synthesis operator $\sy_\eta \colon L^2 (\Sub) \to L^2 (G)$, and frame operator $\ft_\eta \colon L^2 (G) \to L^2 (G)$ associated to $\eta$ over $\Sub$ are all bounded linear operators. This is an attractive property of $\heis{G}{\Sub}$ as a function space in time-frequency analysis, at least when focusing on the subgroup $\Sub$. We also show that elements of the Heisenberg module satisfy the fundamental identity of Gabor analysis over the subgroup $\Sub$ when it is a lattice (\cref{prop:figa}).

We also comment on the assumption in \cref{thm:intro1} that $\Sub$ is cocompact. This is necessary for our localization techniques to work, see \cref{prop:bimodule_localization}. However, as shown in \cite[Theorem 5.1]{JaLe16}, the existence of a multi-window Gabor frame over $\Sub$ implies that the quotient $(\tfp{G})/\Sub$ is compact, i.e.\ $\Sub$ is a cocompact subgroup of $\tfp{G}$. The assumption is therefore mild.

The paper is structured as follows: In \Cref{Section:Preliminaries}, we cover the necessary background material on frames in Hilbert $C^*$-modules, continuous Gabor frames and Heisenberg modules. In \Cref{Section:Main-results}, we introduce the notion of the localization of a Hilbert $C^*$-module with respect to a (possibly unbounded) trace on the coefficient algebra, and compute the localization of the Heisenberg module. We then give applications to Gabor analysis.

\subsection*{Acknowledgements}

The authors would like to thank Nadia Larsen, Franz Luef and Luca Gazdag for giving feedback on a draft of the paper. The second author would like to thank Erik Bédos for helpful discussions. The authors are also indebted to the referees for their invaluable feedback on the first draft of the paper, and to one of the referees for suggesting a simpler proof of \Cref{prop:frame_finite}, which we have included.

\section{Preliminaries}
\label[section]{Section:Preliminaries}

\subsection{Frames in Hilbert \texorpdfstring{$\boldsymbol{C^*}$}{C*}-modules}
\label[subsection]{Section:C*-mods-and-frames}

In the interest of brevity, we will assume basic knowledge about $C^*$-algebras, Hilbert $C^*$-modules, imprimitivity bimodules and adjointable operators between such modules. We mention \cite{RaWi98,La95} as references. Instead, we dedicate this section to introduce module frames.

The $A$-valued inner product of a left Hilbert $A$-module will in general be denoted by $\lhs{\cdot}{\cdot}$, while the $A$-valued inner product of a right Hilbert $A$-module will be denoted by $\rhs{\cdot}{\cdot}$. We often refer to $A$ as the \emph{coefficient algebra} of $\E$. If $\E$ and $\F$ are left Hilbert $A$-modules, we use $\Adj_A(\E, \F)$ to denote the Banach space of adjointable operators $\E \to \F$, or just $\Adj(\E,\F)$ when there is no chance of confusion. As is standard, we write $\Adj(\E) = \Adj_A (\E)$ for the $C^*$-algebra $\Adj_A (\E,\E)$, and $\mathcal{K}(E) = \mathcal{K}_A(E)$ for the (generalized) compact operators on $\E$.

For an (at most) countable index set $J$, we denote by $\ell^2(J,A)$ the left Hilbert $A$-module of all sequences $(a_j)_{j \in J}$ in $A$ for which the sum $\sum_{j \in J} a_ja_j^*$ converges in $A$, with $A$-valued inner product
\[ \lhs{(a_j)_{j \in J}}{(b_j)_{j \in J}} = \sum_{j \in J} a_j b_j^* .\]
There is an analogous way to make $\ell^2(J,A)$ into a \emph{right} Hilbert $A$-module, by replacing $a_j b_j^*$ with $a_j^* b_j$ in the definition. We will work with left modules throughout this section, but obvious modifications can be made for the case of right modules as well.

We now define module frames in Hilbert $A$-modules, introduced in \cite{FrLa02} in the case where $A$ is unital. For a treatment of the possibly non-unital case, see \cite{ArBa17}.
\begin{definition}\label{def:frame}
Let $A$ be a $C^*$-algebra and $\E$ be a left Hilbert $A$-module. Furthermore, let $J$ be some countable index set and let $( \eta_j )_{j\in J}$ be a sequence in $\E$. 
We say $( \eta_j )_{j\in J}$ is a \textit{module frame for $\E$} if there exist constants $C,D > 0$ such that
\begin{equation}
\label{Eq:Modular-frame-ineq}
    C \lhs{\xi}{\xi} \leq \sum_{j \in J} \lhs{\xi}{\eta_j}\lhs{\eta_j}{\xi} \leq D \lhs{\xi}{\xi}
\end{equation}
for all $\xi \in \E$, and the middle sum converges in norm. The constants $C$ and $D$ are called lower and upper frame bounds, respectively.
\end{definition}
\begin{remark}
If $A = \C$ in the above definition then $\E$ is a Hilbert space, and we recover the definition of frames in Hilbert spaces due to Duffin and Schaeffer \cite{DuSc52}.
\end{remark}
\begin{remark}
We will never treat frames over different index sets simultaneously, so to ease notation we will sometimes leave the index set implied.
\end{remark}

Let $( \eta_j )_{j \in J}$ be a sequence in $\E$ that satisfies the upper frame bound condition in \cref{def:frame} but not necessarily the lower frame bound condition. Such a sequence is called a \emph{Bessel sequence} and every constant $D>0$ for which \eqref{Eq:Modular-frame-ineq} is true is called a \emph{Bessel bound} for $( \eta_j )_{j \in J}$. To a Bessel sequence $( \eta_j)_{j \in J}$ we associate the \emph{module analysis operator} $\modan = \modan_{(\eta_j)_j} \colon \E \to \ell^2(J,A)$ given by 
\begin{equation}
    \modan \xi = ( \lhs{\xi}{\eta_j} )_{j \in J}
\end{equation}
for $\xi \in \E$.
It is an adjointable $A$-linear operator, and its adjoint $\modsy = \modsy_{(\eta_j)_j}$ is known as the \textit{module synthesis operator}, and is given by
\begin{equation}
    \modsy ((a_j)_j) = \sum_{j\in J} a_j \cdot \eta_j,
\end{equation}
for $(a_j)_j \in \ell^2(J,A)$. 
Now let $ ( \gamma_j )_{j\in J}$ be another Bessel sequence. We then define the \textit{module frame-like operator} $\modft \in \Adj_A (\E)$ by $\modft = \modft_{( \eta_j )_j , ( \gamma_j )_j} := \modsy_{( \gamma_j )_j} \modan_{( \eta_j )_j}$. That is, for all $\xi \in \E$ we have
\begin{equation}
    \modft \xi = \sum_{j\in J} \lhs{\xi}{\eta_j} \cdot \gamma_j.
\end{equation}
In case $( \eta_j )_j = ( \gamma_j )_j$ we write $\modft_{( \eta_j )_j} := \modft_{( \eta_j )_j, ( \eta_j )_j}$ and call it the \textit{module frame operator} (associated to $( \eta_j )_j$). Since $\modft_{( \eta_j )_j} = \modan_{( \eta_j )_j}^*\modan_{( \eta_j )_j}$, we see that $\modft_{( \eta_j )_j}$ is always a positive operator.

A special case of the above situation is when we consider a sequence $( \eta)$ consisting of a single element $\eta \in \E$, i.e.\ $|J| = 1$. It follows by the Cauchy-Schwarz inequality for Hilbert $C^*$-modules that such a sequence is automatically a Bessel sequence. We write $\modan_{\eta} = \modan_{(\eta)}$, $\modsy_{\eta} = \modsy_{(\eta)}$, $\modft_{\eta,\gamma} = \modft_{(\eta),(\gamma)}$ for another sequence $(\gamma)$ where $\gamma \in \E$, and $\modft_{\eta} = \modft_{(\eta)}$. Note that in this case, $\modan_{\eta} \in \Adj_A(\E,A)$, $\modsy_{\eta} \in \Adj_A(A,\E)$ and $\modft_{\eta,\gamma} \in \Adj_A(\E,\E)$ are given by
\begin{align*}
    \modan_{\eta}\xi &= \lhs{\xi}{\eta} \\
    \modsy_{\eta}{a} &= a \cdot \eta \\
    \modft_{\eta,\gamma} \xi &= \lhs{\xi}{\eta} \cdot\gamma
\end{align*}
for $\xi \in \E$, $a \in A$. Also, for a finite Bessel sequence $(\eta_1, \ldots, \eta_k)$, we have that $\modan_{(\eta_j)_{j=1}^k} = \sum_{j=1}^k \modan_{\eta_j}$, and similar equalities for the synthesis and frame-like operators. The operator $\modft_{\eta,\gamma}$ is often called a \emph{rank-one operator}, and we have the following proposition, which is immediate by \cite[Lemma 2.30, Proposition 3.8]{RaWi98}.

\begin{proposition}\label{prop:rank_one_operator_norm}
Let $\eta$ be an element of a full left Hilbert $A$-module $\E$. Then
\[ \| \eta \|_{\E} = \| \modft_{\eta} \|_{\Adj_A(\E)}  .\]
More generally, if $\mathcal{E}$ is an imprimitivity $A$-$B$-bimodule, then
\[ \| \lhs{\xi}{\eta} \|_{A} = \| \rhs{\eta}{\xi} \|_{B} \]
for every $\xi,\eta \in \E$. Hence, the norm of $\E$ as a left Hilbert $A$-module coincides with the norm of $\E$ as a right Hilbert $B$-module.
\end{proposition}

The frame property of a Bessel sequence $(\eta_j)_{j \in J}$ can be characterized in terms of the invertibility of the associated frame operator $\modft_{(\eta_j)_j}$. For a proof, see \cite[Theorem 1.2]{ArBa17}.

\begin{proposition}\label{prop:frame_invertible}
Let $( \eta_j )_{j \in J}$ be a Bessel sequence in $\E$. Then the frame operator $\modft_{( \eta_j )_j}$ associated to $ ( \eta_j)_j$ is invertible if and only if $( \eta_j )_j$ is a module frame for $\E$. 
\end{proposition}

The following proposition shows that finite module frames are nothing more than (algebraic) generating sets, and conversely.

\begin{proposition}\label{prop:frame_finite}
Let $\E$ be a left Hilbert $A$-module, and let $\eta_1, \ldots, \eta_k \in \mathcal{E}$. Then $(\eta_1, \ldots, \eta_k)$ is a module frame for $\mathcal{E}$ if and only if it is a generating set for $\mathcal{E}$, i.e.\ for every $\xi \in \mathcal{E}$ there exist coefficients $a_1, \ldots, a_k \in A$ such that
\[ \xi = \sum_{j=1}^k a_j \cdot \eta_j .\]
\end{proposition}

\begin{proof}
Let $\modft$ be the module frame operator corresponding to $(\eta_j)_j$. If $(\eta_j)_j$ is a frame for $\E$, then by \cite[Theorem 1.2]{ArBa17} one has the expansion $\xi = \sum_{j=1}^k \lhs{\xi}{\modft^{-1}\eta_j}\cdot \eta_j$ for every $\xi \in \E$. This shows that $(\eta_j)_j$ is a generating set for $\E$.

We now prove the converse. Denote by $\modan: \E \to A^k$ the map $\modan \xi = (\lhs{\xi}{\eta_j})_{j=1}^k$. This is an adjointable $A$-module map, with $\modan^* (a_j)_{j=1}^k = \sum_{j=1}^k a_j \eta_j$. By assumption $\modan^*$ is a surjection. \cite[Theorem 3.2]{La95} then gives that the image of $\modan$ is a complementable submodule of $A^k$. The usual Hilbert space argument then gives that $\modan^* \modan: \E \to \E$ is invertible, and it follows from \cref{prop:frame_invertible} that $(\eta_1, \ldots, \eta_k)$ is a module frame for $\E$.

\end{proof}

\subsection{Gabor analysis on locally compact abelian groups}
\label{Section:Gabor-LCA}

For the rest of the paper (unless stated otherwise), $G$ will denote a second countable, locally compact abelian group with group operation written additively and with identity $0 \in G$, and $\Sub$ will denote a closed subgroup of the time-frequency plane $\tfp{G}$. We fix a Haar measure on $G$ and equip $\widehat{G}$ with the dual measure \cite[Theorem 4.21]{Fo95}. Furthermore, we pick a Haar measure on $\Sub$, and let $(\tfp{G})/\Sub$ have the unique measure such that  Weil's formula holds \cite[equation (2.4)]{JaLe16}. We can then associate to $\Sub$ the quantity $s(\Delta) = \mu ((\tfp{G})/\Sub)$, known as the \textit{size of $\Sub$} \cite[p. 235]{JaLe16}. Here $\mu$ denotes the chosen Haar measure. The size of $\Sub$ is finite precisely when $(\tfp{G})/\Sub$ is compact, that is, $\Sub$ is cocompact in $\tfp{G}$.

Given $x \in G$ and $\omega \in \widehat{G}$, we define the translation operator $T_x$ and modulation operator $M_{\omega}$ on $L^2(G)$ by
\begin{align*}
    (T_x \xi)(t) &= \xi(t-x), & (M_{\omega}\xi)(t) &= \omega(t) \xi(t)
\end{align*}
for $\xi \in L^2(G)$ and $t \in G$. The translation and modulation operators are unitary linear operators on $L^2(G)$. Moreover, a time-frequency shift is an operator of the form $\pi(x,\omega) = M_{\omega}T_x$ for $x \in G$ and $\omega \in \widehat{G}$.

The \emph{adjoint subgroup} of $\Sub$, denoted by $\Sub^{\circ}$, is the closed subgroup of $\tfp{G}$ given by
\begin{align*}
    \Sub^{\circ} &= \{ w \in \tfp{G} : \pi(z)\pi(w) = \pi(w)\pi(z) \; \text{for all} \; z \in \Sub \}.
\end{align*}
We use the identification of $\Sub^{\circ}$ with $((\tfp{G})/\Sub)^{\widehat{}}$ in \cite[p. 234]{JaLe16} to pick the dual measure on $\Sub^{\circ}$ corresponding to the measure on $(\tfp{G})/\Sub$ induced from the chosen measure on $\Sub$. If $\Sub$ is cocompact in $\tfp{G}$, then $\Sub^{\circ}$ is discrete, and the induced measure on $\Sub^{\circ}$ will be the counting measure scaled by the constant $s(\Sub)^{-1}$ \cite[equation (13)]{JaLu18}.

We consider the two following important examples:

\begin{example}\label{ex:lattice}
Suppose $\Sub$ is a lattice in $\tfp{G}$, namely a discrete, cocompact subgroup of $\tfp{G}$. Then $\Sub^{\circ}$ is also a lattice in $\tfp{G}$ \cite[Lemma 3.1]{Ri88}. In this situation, we will usually equip $\Sub$ with the counting measure. The size of $\Sub$ is then the measure of any fundamental domain for $\Sub$ in $\tfp{G}$ \cite[Remark 1]{JaLe16}. Since $\Sub$ in particular is cocompact, the measure on $\Sub^{\circ}$ will not be the counting measure in general, but rather the counting measure scaled by $s(\Sub)^{-1}$.
\end{example}

\begin{example}\label{ex:tfp}
Let $\Sub = \tfp{G}$. $\Sub$ is then cocompact in $\tfp{G}$, since $(\tfp{G})/\Sub$ is trivial. The natural choice of measure on $\Sub$ in this situation is the product measure coming from the chosen measure on $G$ and the dual measure on $\widehat{G}$. The induced measure on $\Sub^{\circ} = \{ 0 \}$ is then the normalized measure assigning the value 1 to $\{0 \}$.
\end{example}

\subsection{Gabor frames.}
We will need a continuous version of Gabor frames, and so we cannot treat our Gabor frames as a special case of \cref{def:frame}. However, note the similarities between the definitions and results given here and in \cref{Section:C*-mods-and-frames}.

Given $\eta \in L^2(G)$, the family $\mathcal{G}(\eta;\Sub) = ( \pi(z) \eta)_{z \in \Sub}$ is called a \emph{Gabor system} over $\Sub$ with generator $\eta$. More generally, given $\eta_1, \ldots, \eta_k \in L^2(G)$, the family $\mathcal{G}(\eta_1, \ldots, \eta_k ; \Sub) = (\pi(z) \eta_j)_{z \in \Sub, 1 \leq j \leq k}$ is called a \emph{multi-window Gabor system} over $\Sub$ with generators $\eta_1, \ldots, \eta_k$.

The multi-window Gabor system $\mathcal{G}(\eta_1, \ldots, \eta_k ; \Sub)$ is called a \emph{multi-window Gabor frame} if it is a \emph{(continuous) frame} \cite{AlAn93, Ka94, JaLe16} for $L^2(G)$ in the sense that both of the following hold:
\begin{enumerate}
    \item The family $\mathcal{G}(\eta_1, \ldots, \eta_k; \Sub)$ is weakly measurable, that is, for every $\xi \in L^2(G)$ and each $1 \leq j \leq k$, the map $z \mapsto \langle \xi, \pi(z) \eta_j \rangle$ is measurable.
    \item There exist positive constants $C,D > 0$ such that for all $\xi \in L^2(G)$ we have that
\begin{equation}\label{eq:multi-frame-ineq}
    C \| \xi \|_2^2 \leq \sum_{j=1}^k \int_{\Sub} | \langle \xi, \pi(z) \eta_j \rangle |^2 \, \dif z \leq D \| \xi \|_2^2 .
\end{equation}
\end{enumerate}
The numbers $C$ and $D$ are called lower and upper frame bounds respectively. We may also refer to the upper frame bound as a \emph{Bessel bound} in analogy with \Cref{Section:Preliminaries}. If the family $\mathcal{G}(\eta_1,\ldots, \eta_k;\Sub)$ is weakly measurable and has an upper frame bound but not necessarily a lower frame bound, we call it a \emph{Bessel family}. A (single-window) Gabor system which is a frame is called a \emph{Gabor frame}.

The \emph{analysis operator} associated to a Bessel family $( \pi(z) \eta)_{z \in \Sub}$ is the bounded linear operator $\an_{\eta} \colon L^2(G) \to L^2(\Sub)$ given by
\begin{equation}
    \an_{\eta} \xi = ( \langle \xi, \pi(z) \eta \rangle)_{z \in \Sub}
\end{equation}
for $\xi \in L^2(G)$. Its adjoint $\sy_{\eta} \colon L^2(\Sub) \to L^2(G)$ is called the \emph{synthesis operator} and is given weakly by
\begin{equation}
    \sy_{\eta} (c_z)_{z \in \Sub} = \int_{\Sub} c_z \pi(z) \eta \, \dif z
\end{equation}
for $(c_z)_{z \in \Sub} \in L^2(\Sub)$. The \emph{frame-like operator} associated to two Bessel families $\mathcal{G}(\eta;\Sub)$ and $\mathcal{G}(\gamma;\Sub)$ is the operator $S_{\eta,\gamma} = D_{\gamma}C_{\eta}$ which is given weakly by
\begin{equation}
\label{Eq:frame-like-operator}
    \ft_{\eta,\gamma} \xi = \int_{\Sub} \langle \xi, \pi(z) \eta \rangle \pi(z) \gamma \, \dif z
\end{equation}
for $\xi \in L^2(G)$. In particular, the \emph{frame operator} associated to the Bessel family $\mathcal{G}(\eta;\Sub)$ is the operator $\ft_{\eta} := \ft_{\eta,\eta}$. This is a positive operator.

If $\mathcal{G}(\eta_1, \ldots, \eta_k ; \Sub)$ is a multi-window Gabor Bessel family, then its analysis, synthesis and frame operators are given respectively by $\an= \sum_{j=1}^k \an_{\eta_j}$, $\sy = \sum_{j=1}^k \sy_{\eta_j}$ and $\ft = \sum_{j=1}^k \ft_{\eta_j}$.

Note how the following proposition is analogous to \cref{prop:frame_invertible}. The result is well-known in frame theory.

\begin{proposition}
\label{prop:Gabor-ft-inv-iff-Gabor-system}
Let $\eta_1, \ldots, \eta_k \in L^2(G)$ be such that $\mathcal{G}(\eta_1, \ldots, \eta_k ; \Sub)$ is a Bessel family for $L^2(G)$. Then $\mathcal{G}(\eta_1, \ldots, \eta_k ; \Sub)$ is a multi-window Gabor frame if and only if the associated frame operator $S = \sum_{j=1}^k S_{\eta_j}$ is invertible on $L^2(G)$.
\end{proposition}

%

The \emph{Feichtinger algebra} $S_0(G)$ is the set of $\xi \in L^2(G)$ for which
\begin{equation}
    \int_{\tfp{G}} | \langle \xi, \pi(z) \xi \rangle | \, \dif z < \infty .
\end{equation}
See \cite{Ja18} for a comprehensive introduction to $S_0(G)$. For us, the Feichtinger algebra will play a key role in the construction of Heisenberg modules as in \cite{Lu09}, see \cref{prop:heisenberg_module}. Note that in the original paper \cite{Ri88}, the Schwartz-Bruhat space $\mathcal{S}(G)$ was used instead. The Schwartz-Bruhat space has a more technical definition. Although it will not be important to us, we mention that the Feichtinger algebra has a natural Banach space structure \cite[Theorem 1]{Fe81}. 

\begin{proposition}\label{prop:feichtinger_algebra}
The following properties hold for the Feichtinger algebra:
\begin{enumerate}
    \item If $\eta \in S_0(G)$, then $\mathcal{G}(\eta;\Sub)$ is a Bessel family for $L^2(G)$. \label{prop:feichtinger_algebra_it1}
    \item If $G$ is discrete, then $S_0(G) = \ell^1(G)$. \label{prop:feichtinger_algebra_it2}
\end{enumerate}
\end{proposition}

For a proof of these results, see \cite[Corollary A.5]{JaLe16} and \cite[Lemma 4.11]{Ja18}.

\subsection{Twisted group \texorpdfstring{$\boldsymbol{C^*}$}{C*}-algebras and Heisenberg modules}
\label{Section:Twisted-algs-and-Heisenberg}

For the moment, let $\Sub$ be a general second countable, locally compact abelian group. A \emph{(normalized) continuous 2-cocycle} on $\Sub$ is a continuous map $c \colon \Sub \times \Sub \to \T$ that satisfies the following two identities:
\begin{enumerate}
    \item For every $z_1,z_2,z_3 \in \Sub$ we have that
    \begin{equation}
        c(z_1,z_2)c(z_1+z_2,z_3) = c(z_1,z_2+z_3)c(z_2,z_3).
    \end{equation}
    \item If $0$ denotes the identity element of $\Sub$, then
    \begin{equation}
        c(0,0) = 1.
    \end{equation}
\end{enumerate}
Note that if $c$ is a continuous 2-cocycle, then its pointwise complex conjugate $\overline{c}$ is a continuous 2-cocycle as well.

Given a continuous 2-cocycle $c$ on $\Sub$, one can equip the Feichtinger algebra  $S_0(\Sub)$ with a multiplication and involution as follows: For $a,b \in S_0(\Sub)$ and $z \in \Sub$, one defines
\begin{align}
    a * b(z) &= \int_{\Sub} c(w,z-w) a(w) b(z-w) \, \dif w \\
    a^*(z) &= \overline{c(z,-z) a(-z)} .
\end{align}

The $C^*$-enveloping algebra of $S_0(\Sub,c)$ is called the \emph{$c$-twisted group $C^*$-algebra of $\Delta$} and is denoted by $C^*(\Delta,c)$. Note that this definition is equivalent to the usual definition of $C^*(\Sub,c)$ as the $C^*$-enveloping algebra of $L^1(\Sub,c)$, as $S_0(\Sub,c)$ is dense in $L^1(\Sub,c)$ and the $L^1$-norm dominates the universal $C^*$-norm on $L^1(\Sub,c)$.

Let $H$ be a Hilbert space, and denote by $\mathcal{U}(H)$ the unitary operators on $H$. A map $\pi \colon \Sub \to \mathcal{U}(H)$ is called a \emph{$c$-projective unitary representation of $\Sub$ on $H$} if the following two properties hold:
\begin{enumerate}
    \item $\pi$ is strongly continuous, i.e.\ for every $\xi \in H$, the map $\Sub \to H$, $z \mapsto \pi(z) \xi$ is continuous.
    \item For every $z,w \in \Sub$, we have that
    \begin{equation}
        \pi(z)\pi(w) = c(z,w) \pi(z+w) . \label{eq:projective_rep}
    \end{equation}
\end{enumerate}
The twisted group $C^*$-algebra $C^*(\Sub,c)$ captures the $c$-projective unitary representation theory of $\Sub$ in the following sense: For every $c$-projective unitary representation $\pi \colon \Sub \to \mathcal{U}(H)$ on a Hilbert space $H$, there is a nondegenerate $*$-representation $\overline{\pi} \colon C^*(\Sub,c) \to \Adj(H)$ which for $a \in L^1(\Sub,c)$ is given weakly by
\begin{equation}
    \overline{\pi}(a) = \int_{\Sub} a(z) \pi(z) \, \dif z .
\end{equation}
The above representation is called the \emph{integrated} representation of $\pi$. Conversely, if $\Pi : C^*(\Sub,c) \to \Adj(H)$ is any nondegenerate $*$-representation of $C^*(\Sub,c)$ on $H$, then there exists a unique $c$-projective unitary representation $\pi \colon \Sub \to \mathcal{U}(H)$ such that $\overline{\pi} = \Pi$. This correspondence can be seen as a consequence of e.g.\ \cite[Proposition 2.7]{PaRe89}.

Note also that if $\pi$ is a $c$-projective unitary representation, then $\pi^*$ defined by $\pi^*(z) = \pi(z)^*$ is $\overline{c}$-projective. This follows from taking the adjoint of both sides of \eqref{eq:projective_rep} (it is essential that we are working with abelian groups in this situation).


When $\Sub$ is discrete, we have by \cref{prop:feichtinger_algebra} \ref{prop:feichtinger_algebra_it2} that $S_0(\Sub,c) \cong \ell^1(\Sub,c)$. If we equip $\Sub$ with the counting measure, there is a canonical tracial state on $C^*(\Sub,c)$ \cite[p. 951]{BeOm18}. On the dense $*$-subalgebra $\ell^1(\Sub,c)$, it is given by
\begin{equation}
    \tr(a) = a(0) \label{eq:tr_twist}
\end{equation}
for $a \in \ell^1(\Sub,c)$.

We now return to the situation where $G$ is a second countable, locally compact abelian group, and $\Sub$ is a closed subgroup of $\tfp{G}$. The map $c \colon \Sub \times \Sub \to \T$ given by
\begin{equation}
    c((x,\omega),(y,\tau)) = \overline{\tau(x)}
\end{equation}
for $(x,\omega),(y,\tau) \in \Sub$ is a continuous 2-cocycle on $\Sub$ called the \emph{Heisenberg 2-cocycle} \cite[p. 263]{Ri88}. Moreover, the time-frequency shifts $\pi(x,\omega) = M_{\omega} T_x$ define a $c$-projective unitary representation of $G \times \widehat{G}$ on $L^2(G)$, and so we have that
\[ \pi(x,\omega)\pi(y,\tau) = \overline{\tau(x)} \pi(x+y,\omega\tau) .\]
This representation is often called the \emph{Heisenberg representation}. Restricting to the closed subgroup $\Sub$ of $\tfp{G}$, we obtain a $c$-projective unitary representation of $\Sub$ on $L^2(G)$. We denote the restriction by $\pi_{\Delta}$. This representation then induces a $*$-representation of $C^*(\Sub,c)$ on $L^2(G)$, which we also (by slight abuse of notation) denote by $\pi_\Sub$. We have the following result, see \cite[Proposition 2.2]{Ri88}.

\begin{proposition}\label{prop:faithful_rep}
The integrated representation $\pi_{\Delta} \colon C^*(\Sub,c) \to \Adj(L^2(G))$ is faithful, i.e.\ $\pi_{\Delta}(a) = 0$ implies $a=0$ for all $a \in C^*(\Sub,c)$.
\end{proposition}

In the following proposition, we give the definition of Heisenberg modules. For a proof, see the proof of \cite[Theorem 3.4]{JaLu18} or Rieffel's arguments from \cite{Ri88} which are similar.

\begin{proposition}\label{prop:heisenberg_module}
Let $G$ be a locally compact abelian group, and let $\Sub$ be a closed subgroup of $\tfp{G}$, both with chosen Haar measures. Equip $\Sub^{\circ}$ with the Haar measure determined as in \cref{Section:Gabor-LCA}. The \emph{Heisenberg module} $\heis{G}{\Sub}$ is an imprimitivity $C^*(\Sub,c)$-$C^*(\Sub^{\circ},\overline{c})$-module obtained as a completion of the Feichtinger algebra $S_0(G)$. The actions and inner products are given densely as follows:
\begin{enumerate}
    \item If $a \in S_0(\Sub,c)$, $b \in S_0(\Sub^{\circ}, \overline{c})$ and $\xi \in S_0(G)$, then $a \cdot \xi, \xi \cdot b \in S_0(G)$, with
    \begin{align}
        a \cdot \xi = \int_{\Sub} a(z) \pi(z) \xi \, \dif z, && \xi \cdot b = \int_{\Sub^{\circ}} b(w) \pi (w)^* \xi \dif w . \label{eq:mod_action}
    \end{align}
    \item If $\xi,\eta \in S_0(G)$, then $\lhs{\xi}{\eta} \in S_0(\Sub,c)$ and $\rhs{\xi}{\eta} \in S_0(\Sub^{\circ},\overline{c})$, with
    \begin{align}
        \lhs{\xi}{\eta}(z) = \langle \xi, \pi(z) \eta \rangle, && \rhs{\xi}{\eta}(w) =   \hs{\pi (w) \eta}{\xi} \label{eq:mod_inner_prod}
    \end{align}
    for $z \in \Sub$ and $w \in \Sub^{\circ}$.
\end{enumerate}
\end{proposition}


\begin{sloppypar}
We can rewrite the left and right actions of \cref{prop:heisenberg_module} as follows: Since ${\pi \colon G \times \widehat{G} \to \Adj(L^2(G))} $  is a $c$-projective unitary representation, it follows that $\pi^*$ is $\overline{c}$-projective. We restrict $\pi$ and $\pi^*$ to $\Sub$ and $\Sub^{\circ}$ respectively and obtain the representations $\pi_{\Delta}$ and $\pi^*_{\Delta^{\circ}}$. Passing to the integrated representations, we obtain $*$-representations of $C^*(\Sub,c)$ and $C^*(\Sub^{\circ},\overline{c})$ which we also denote by $\pi_{\Sub}$ and $\pi_{\Sub^{\circ}}^*$ respectively. We can then write the left and right module actions given in \eqref{eq:mod_action} as 
\end{sloppypar}
\begin{align}
    a \cdot \xi = \pi_{\Sub}(a)\xi, && \xi \cdot b = \pi_{\Sub^{\circ}}^*(b)\xi \label{eq:inner_prod_rep}
\end{align}
for $\xi \in S_0(G)$, $a \in S_0(\Sub,c)$ and $b \in S_0(\Sub^{\circ},\overline{c})$.

\section{Results}
\label[section]{Section:Main-results}

\subsection{Localization of Hilbert \texorpdfstring{$\boldsymbol{C^*}$}{C*}-modules.}
We will use localization of Hilbert $C^*$-modules with respect to positive linear functionals as defined in \cite[p.\ 7]{La95}. Localization is a technique reminiscent of the GNS construction. It uses a positive linear functional on the coefficient algebra of a Hilbert $C^*$-module to embed the module continuously into a Hilbert space. The authors are not aware of many uses of localization in the literature, but an example is found in \cite{KaLe12}. We will focus exclusively on the case of faithful traces, but we will need a version for (possibly) unbounded traces, which we develop after reviewing the case of finite faithful traces.

Let $\tr : A \to \C$ denote a finite trace on $A$, i.e.\ a positive linear functional on $A$ that satisfies $\tr(a^*a) = \tr(aa^*)$ for all $a \in A$. Assume also that $\tr$ is faithful, that is, $\tr(a^*a) = 0$ implies $a=0$ for all $a \in A$. If $\E$ is a left Hilbert $A$-module, it is easily verified that
\begin{equation}
    \langle \xi, \eta \rangle_{\tr} = \tr( \lhs{\xi}{\eta} )
\end{equation}
for $\xi,\eta \in \E$ defines a ($\C$-valued) inner product on $\E$, and we denote the Hilbert space completion of $\E$ in the norm $\| \cdot \|_{H_{\E}}$ coming from $\langle \cdot, \cdot \rangle_{\tr}$ by $H_{\E}$. For $\xi \in \E$, the chain of inequalities
\[ \| \xi \|^2_{H_{\E}} = \tr( \lhs{\xi}{\xi} ) \leq \| \tr \| \| \lhs{\xi}{\xi} \|_A = \| \tr \| \| \xi \|_{\E}^2 \]
shows that the embedding $\E \hookrightarrow H_{\E}$ is continuous. Moreover, if $\tr$ is a state, that is, $\| \tr \|=1$, then the embedding is norm-decreasing. The Hilbert space $H_{\E}$ is called the \emph{localization} of $\E$ with respect to $\tr$.

If $\E$ and $\F$ are left Hilbert $A$-modules, we obtain localizations $H_{\E}$ and $H_{\F}$ with respect to $\tr$. Let $T \colon \E \to \F$ be an adjointable linear operator. Then in particular, $T$ is a bounded linear operator when viewing the Hilbert $C^*$-modules as Banach spaces, and we denote its norm by $\Vert T \Vert$. For all $\xi \in \E$ we have that $\lhs{T\xi}{T\xi} \leq \Vert T \Vert^2 \lhs{\xi}{\xi}$ \cite[Corollary 2.22]{RaWi98}. Applying $\tr$ on both sides, we obtain
\begin{equation}
    \Vert T\xi \Vert^2_{H_{\F}} \leq \Vert T \Vert^2 \Vert \xi \Vert^2_{H_{\E}}, \label{eq:bounded_hilbert}
\end{equation}
which shows that $T$ extends to a bounded linear operator of Hilbert spaces $\overline{T} \colon H_{\E} \to H_{\F}$. If $\| \overline{T} \|_h$ denotes the norm of $\overline{T}$ as a Hilbert space operator, then \eqref{eq:bounded_hilbert} also shows that $\| \overline{T} \|_h \leq \Vert T \Vert $. Hence we have a norm-decreasing (hence continuous) inclusion of Banach spaces $\Adj (\E, \F) \longrightarrow \Adj (H_{\E}, H_{\F})$. If $\E = \F$, then more is true: We obtain an injective $*$-homomorphism of $C^*$-algebras \cite[p. 58]{La95} $\Adj(\E) \longrightarrow \Adj (H_{\E})$. Since injective $*$-homomorphisms of $C^*$-algebras are necessarily isometries \cite[Theorem 3.1.5]{Mu90}, we deduce that $\Adj(\E) \to \Adj (H_{\E})$ is an isometry. Hence in this case we have
\begin{equation}
    \| \overline{T} \|_h = \Vert T \Vert \label{eq:operator_norm_equality}
\end{equation}
for all $T \in \Adj (\E)$.

We can define the localization of a \emph{right} Hilbert $A$-module $\E$ at a faithful trace $\tr$ similarly, except in this situation we have to set the inner product to be $\langle \xi, \eta \rangle_{\tr} = \tr( \rhs{\eta}{\xi})$ for $\xi,\eta \in \E$ to get linearity in the first argument instead of the second. Just as with left modules, we obtain a Hilbert space $H_{\E}$ together with an injective bounded linear map $\E \hookrightarrow H_{\E}$.\\

In the following, we develop a version of localization with respect to a possibly unbounded trace that works for our purposes. Denote by $A_+$ the positive elements of the $C^*$-algebra $A$. By a \emph{weight} on $A$, we will mean a function $\phi : A_+ \to [0,\infty]$ that satisfies $\phi(a+b) = \phi(a)+\phi(b)$ for all $a,b \in A_+$, $\phi(\lambda a) = \lambda \phi(a)$ for all $a \in A_+$ and $\lambda > 0$, and $\phi(0) = 0$. The weight $\phi$ is \emph{lower semi-continuous} if whenever $(a_\alpha)_{\alpha}$ is a net in $A_+$ converging to $a$, then $\phi(a) \leq \liminf_{\alpha} \phi(a_\alpha)$. A weight $\phi$ on $A$ is a \emph{trace} if $\phi(a^*a) = \phi(aa^*)$ for all $a \in A$, and is \emph{faithful} if $\phi(a) = 0$ implies $a=0$ for every $a \in A_+$.

For a weight $\phi$ on $A$, let $A_+^{\phi} = \{ a \in A_+ : \phi(a) < \infty \}$. The weight $\phi$ is called \emph{densely defined} if $A_+^{\phi}$ is dense in $A_+$ (in the norm topology). Moreover, let $A^{\phi} = \spn A_{+}^{\phi}$. By \cite[Lemma 5.1.2]{Pe79}, $\phi$ has a unique extension to a positive linear functional on $A^{\phi}$, and $\phi$ is densely defined if and only if $A^{\phi}$ is dense in $A$. A weight $\phi$ on $A$ is called \emph{finite} if $A_+^{\phi} = A_+$. In that case, $\phi$ extends uniquely to a positive linear functional on $A^{\phi} = \spn A_+^{\phi} = \spn A_+ = A$, and so we obtain a positive linear functional on the whole of $A$. Conversely, any positive linear functional on $A$ restricts to a finite weight on $A_+$. If $A$ is a unital $C^*$-algebra, then $\phi$ is finite if and only if $1 \in A_+^{\phi}$ if and only if $\phi$ is densely defined.

Now let $\E$ be a left Hilbert $A$-module, and $\tr$ a (possibly unbounded) trace on $A$. There are two problems with localizing $\E$ with respect to $A$: The first one is that $\tr(\lhs{\xi}{\eta})$ might not be finite for $\xi,\eta \in \E$, which means that we do not get a well-defined inner product by setting $\langle \xi, \eta \rangle = \tr(\lhs{\xi}{\eta})$. The other problem is that we might not get a continuous embedding $\E \to H_{\E}$ even if the inner product is well-defined. However, the following set-up is sufficient for our purposes, and solves the aforementioned problems. The essential ingredient in the proof is a result due to Combes and Zettl \cite{CoZe83}.

\begin{proposition}\label{prop:bimodule_localization}
Let $A$ and $B$ be $C^*$-algebras, and suppose $\tr_B$ is a faithful finite trace on $B$. Then the following hold:
\begin{enumerate}
    \item If $\E$ is an imprimitivity $A$-$B$-bimodule, then there exists a unique lower semi-continuous trace $\tr_A$ such that
\begin{equation}
    \tr_A(\lhs{\xi}{\xi}) = \tr_B(\rhs{\xi}{\xi}) \label{eq:trace_compatible_unbounded}
\end{equation}
for all $\xi \in \E$. Moreover, $\tr_A$ is faithful and densely defined, with $\spn \{ \lhs{\xi}{\eta} : \xi,\eta \in \E \} \subseteq A^{\tr_A}$, and setting
\begin{equation}
    \langle \xi, \eta \rangle_{\tr_A} = \tr_A(\lhs{\xi}{\eta}) \label{eq:inner_product_unbounded}
\end{equation}
for $\xi,\eta \in \E$ defines an inner product on $\E$, with $\langle \xi, \eta \rangle_{\tr_A} = \langle \xi, \eta \rangle_{\tr_B}$ for all $\xi,\eta \in \E$. Consequently, the Hilbert space obtained by completing $\E$ in the norm $\| \xi \|' = \tr_A(\lhs{\xi}{\xi})^{1/2}$ is just the localization of $\E$ with respect to $\tr_B$.
    \item If $\E$ and $\F$ are imprimitivity $A$-$B$-bimodules, then every adjointable $A$-linear operator $\E \to \F$ has a unique extension to a bounded linear operator $H_{\E} \to H_{\F}$. Furthermore, the map $\Adj_A(\E,\F) \to \Adj(H_{\E},H_{\F})$ given by sending $T$ to its unique extension is a norm-decreasing linear map of Banach spaces. Finally, if $\E = \F$, the map $\Adj_A(\E) \to \Adj(H_{\E})$ is an isometric $*$-homomorphism of $C^*$-algebras. 
\end{enumerate}
\end{proposition}

\begin{proof}
Suppose $\E$ is an imprimitivity $A$-$B$-bimodule. By \cite[Proposition 2.2]{CoZe83}, there is a unique lower semi-continuous trace $\tr_A$ on $A$ such that the relation in equation \eqref{eq:trace_compatible_unbounded} holds for all $\xi \in \E$. Since $\tr_B$ is finite, it is densely defined, and so $\tr_A$ is densely defined by the same proposition. The same goes for faithfulness. Since $\tr_A(\lhs{\xi}{\xi}) = \tr_B(\rhs{\xi}{\xi}) < \infty$, we have that $\spn \{ \lhs{\xi}{\xi} : \xi \in \E \} \subseteq \spn A_+^{\tr_A} = A^{\tr_A}$. By the polarization identity for Hilbert $C^*$-modules, elements of the form $\lhs{\xi}{\eta}$ are in $\spn \{ \lhs{\xi}{\xi} : \xi \in \E \}$, and so the unique extension of $\tr_A$ to a positive linear functional on $A^{\tr_A}$ is defined on all elements of the form $\lhs{\xi}{\eta}$ with $\xi,\eta \in \E$. Thus, in this situation the inner product proposed in \eqref{eq:inner_product_unbounded} is well-defined. Again by the polarization identity, the relation in \eqref{eq:trace_compatible_unbounded} implies that $\tr_A(\lhs{\xi}{\eta}) = \tr_B(\rhs{\eta}{\xi})$ for all $\xi, \eta \in \E$, and so $\langle \xi, \eta \rangle_{\tr_A} = \langle \xi, \eta \rangle_{\tr_B}$.

If $T \in \Adj(\E,\F)$, then we have that $\lhs{T\xi}{T\xi} \leq \| T \| \lhs{\xi}{\xi}$ for every $\xi \in \E$. Taking the trace $\tr_A$, we obtain that $\| T \xi \|_{H_{\E}} \leq \| T \| \| \xi \|_{H_{\F}}$, just as in the discussion of localization with respect to finite traces. This shows that $T$ extends to a bounded linear map $H_{\E} \to H_{\F}$, and that the inclusion $\Adj(\E,\F) \to \Adj(H_{\E},H_{\F})$ is norm-decreasing. In particular, if $\E=\F$, it becomes an isometric $*$-homomorphism of $C^*$-algebras.
\end{proof}

We will refer to the localization of $\E$ with respect to $\tr_B$ in \cref{prop:bimodule_localization} above also as the localization of $\E$ with respect to $\tr_A$.

\begin{remark}\label{rmk:both_unital}
If both $A$ and $B$ are unital in \cref{prop:bimodule_localization}, then $\tr_A$, being a densely defined trace on a unital $C^*$-algebra, has to be finite. In that case, we can localize $\E$ as a left $A$-module with respect to $\tr_A$ in the usual fashion, and then \cref{prop:bimodule_localization} tells us that the localization is exactly the same as when done with respect to $\tr_B$. 
\end{remark}

\subsection{Localization of the twisted group \texorpdfstring{$\boldsymbol{C^*}$}{C*}-algebra}

The following proposition shows that for a discrete group $\Sub$ with a 2-cocycle $c$, the localization of $C^*(\Sub,c)$ as a left Hilbert module over itself with respect to the canonical trace can be identified in a natural way with $\ell^2(\Sub)$.

\begin{proposition}\label{prop:loc_twist}
Let $\Sub$ be a discrete group equipped with the counting measure and a 2-cocycle $c$. Denote by $H$ the localization of $C^*(\Sub,c)$ as a left module over itself with respect to its canonical faithful tracial state. Then $H$ can be identified with $\ell^2(\Sub)$ in such a way that the following diagram of inclusions commutes:
\begin{center}
	\begin{tikzpicture}[scale=1.0]
	\node (1) at (0,2) {$\ell^1 (\Sub)$};
	\node (2) at (3,2) {$C^* (\Sub,c)$};
	\node (3) at (0,0) {$\ell^2 (\Sub)$};
	\node (4) at (3,0) {$H$};
	\path[right hook->,font=\scriptsize]
	(1) edge[auto] node[auto] {} (2)
	(1) edge node[auto,swap] {} (3)
	(2) edge node[auto] {} (4);
	\path[->,font=\scriptsize]
	(3) edge node[auto] {$\cong$} (4);
	\end{tikzpicture}
\end{center}
Moreover, the inclusion map $C^*(\Sub,c) \to \ell^2(\Sub)$ is norm-decreasing, that is, for all $a \in C^*(\Sub,c)$ we have that
\[ \| a \|_{\ell^2(\Sub)} \leq \| a \|_{C^*(\Sub,c)} .\]
\end{proposition}

\begin{proof}
We have that $C^*(\Sub,c)$ is dense in $H$ in the Hilbert space norm on $H$, and that $\ell^1(\Sub)$ is dense in $C^*(\Sub,c)$ in the $C^*$-norm on $C^*(\Sub,c)$. Since the $C^*$-norm on $C^*(\Sub,c)$ dominates the Hilbert space norm of $H$, we get that $\ell^1(\Sub)$ is also dense in $H$ in the Hilbert space norm. Moreover $\ell^1(\Sub)$ is also dense in $\ell^2(\Sub)$ in the $\ell^2$-norm.

Denote by $\langle \cdot, \cdot \rangle$ the inner product on $\ell^2(\Sub)$. The $C^*(\Sub,c)$-valued inner product on $C^*(\Sub,c)$ as a left Hilbert $C^*$-module over itself is given by $\lhs{a}{b} = ab^*$ for $a,b \in C^*(\Sub,c)$, and so the inner product with respect to $\tr$ is given by $\langle a, b \rangle_{\tr} = \tr(ab^*)$. If $a,b \in \ell^1(\Sub,c)$, then
\begin{align*}
    \langle a, b \rangle_{\tr} &= \tr(a b^*) 
    = (a b^*)(0) 
    = \sum_{z \in \Sub} c(w,0-w) a(w) b^*(0-w)  \\
    &= \sum_{z \in \Sub} c(w,-w) a(w) \overline{c(-w,w) b(w)} 
    = \sum_{z \in \Sub} a(w) \overline{b(w)}
    = \langle a, b \rangle.
\end{align*}
This shows that $\langle \cdot, \cdot \rangle_{\tr}$ and $\langle \cdot, \cdot \rangle$ agree on the subspace $\ell^1(\Sub,c)$ which is dense in both of the Hilbert spaces as argued. It follows that $H$ can be identified with $\ell^2(\Sub)$ in such a way that the inclusions of $\ell^1(\Sub)$ into $\ell^2(\Sub)$ and $C^*(\Sub,c)$ are preserved. Moreover, since $\tr$ is a state, we have that the inclusion $C^*(\Sub,c) \hookrightarrow \ell^2(G)$ is norm-decreasing.
\end{proof}

\begin{remark}\label{rmk:not_state}
In the sequel the following situation will be relevant: Let $\Sub$ be a discrete group, and denote by $\mu$ the counting measure on $\Sub$. Let $k > 0$ be a constant. Then we can consider the $C^*$-algebra $C^*(\Sub,c)$ defined with respect to the measure $k \mu$ rather than $\mu$, and so all sums involved in formulas for convolutions and norms will have a factor of $k$ in front. In this situation there is still a faithful trace $\tr$ on $C^*(\Sub,c)$ given by $\tr(a) = a(0)$ for $a \in \ell^1(\Sub,c)$. However, note that this is not a state when $k \neq 1$. Indeed, the multiplicative identity of $C^*(\Sub,c)$ is $k^{-1} \delta_0$ rather than $\delta_0$, and so
\[ \tr(1) = \tr(k^{-1} \delta_0) = k^{-1} \delta_0(0) = k^{-1} .\]
If we rescale $\tr$ by $k$, we obtain a state. 
\end{remark}

\subsection{Localization of the Heisenberg module}

We will need a trace on the left $C^*$-algebra $A=C^*(\Sub,c)$ of the Heisenberg module in \cref{prop:heisenberg_module}. When $\Sub$ is a lattice in $\tfp{G}$, we will just consider the canonical faithful trace $\tr_A$ on $C^*(\Sub,c)$. Note that by \cref{prop:bimodule_localization} and \cref{rmk:both_unital}, there exists a finite faithful trace on the right $C^*$-algebra $B=C^*(\Sub^{\circ},\overline{c})$ such that $\tr_A(\lhs{\xi}{\eta}) = \tr_B(\rhs{\eta}{\xi})$ for all $\xi,\eta \in \heis{G}{\Sub}$. If $\xi,\eta \in S_0(G)$, then
\[ \langle \xi, \eta \rangle = \langle \xi, \pi(0)\eta \rangle = \lhs{\xi}{\eta}(0) = \tr_A( \lhs{\xi}{\eta}) = \tr_B( \rhs{\eta}{\xi} ) .\]
But there is a canonical trace $\tr_B'$ on $B$ such that $\tr_B'(b) = b(0)$ whenever $b \in \ell^1(\Sub^{\circ},\overline{c})$. Since $\tr_B'(\rhs{\eta}{\xi}) = \rhs{\xi}{\eta}(0) = \langle \pi(0) \xi, \eta \rangle = \langle \xi, \eta \rangle$, this shows that $\tr_B$ and $\tr_B'$ agree on $\spn \{ \rhs{\xi}{\eta} : \xi,\eta \in S_0(G) \}$. Since the latter is dense in $B$, we conclude that $\tr_B = \tr_B'$. Note however by \cref{rmk:not_state} that the faithful trace $\tr_B$ which satisfies \eqref{eq:inner_product_unbounded} is not a state unless $s(\Sub) =1$.

In the case when $\Sub$ is only cocompact and not necessarily discrete, $\Sub^\circ$ is discrete, and we obtain a (possibly unbounded) trace on $C^*(\Sub,c)$ by the following proposition. Note that we use the measures as chosen in the beginning of this section, and that $B$ is equipped with the canonical trace that is not a state in general.

\begin{proposition}\label{prop:induced_trace_heisenberg_module}
Let $G$ be a second countable, locally compact abelian group, and let $\Sub$ be a closed, cocompact subgroup of $\tfp{G}$. Let $A = C^*(\Sub,c)$ and $B = C^*(\Sub^{\circ},\overline{c})$. Denote by $\tr_B$ the canonical faithful trace on $B$ as in \cref{rmk:not_state}. Then the induced trace $\tr_A$ on $A$ via the Heisenberg module $\heis{G}{\Sub}$ as in \cref{prop:bimodule_localization} is given by
\[ \tr_A(\lhs{\xi}{\eta}) = \langle \xi, \eta \rangle \]
for $\xi,\eta \in S_0(G)$. In particular, if $\Sub$ is a lattice in $\tfp{G}$, then $\tr_A$ is the canonical faithful tracial state on $C^*(\Sub,c)$. 
\end{proposition}

\begin{proof}
By \cref{prop:bimodule_localization}, the induced trace $\tr_A$ satisfies
\[ \tr_A(\lhs{\xi}{\eta}) = \tr_B(\rhs{\eta}{\xi}) \]
for all $\xi,\eta \in \heis{G}{\Sub}$. If $\xi,\eta \in S_0(G)$, then $\rhs{\eta}{\xi} \in S_0(\Sub^{\circ}, \overline{c}) = \ell^1(\Sub^{\circ},\overline{c})$ by \cref{prop:heisenberg_module} and \cref{prop:feichtinger_algebra} \ref{prop:feichtinger_algebra_it2}, and so
\[ \tr_B(\rhs{\eta}{\xi}) = \rhs{\eta}{\xi}(0) = \langle \pi(0)\xi, \eta \rangle = \langle \xi, \eta \rangle .\]
If $\Sub$ is a lattice, then $A$ is the twisted group $C^*$-algebra of a discrete group, and in this case we know that the canonical faithful tracial state $\tr$ on $C^*(\Sub,c)$ is given by $\tr(a) = a(0)$ for $a \in \ell^1(\Sub,c) = S_0(\Sub,c)$. In particular, $\tr(\lhs{\xi}{\eta}) = \langle \xi, \eta \rangle$. By fullness of $\E$ as a left Hilbert $A$-module, it follows that $\tr$ and $\tr_A$ agree on a dense subspace of $A$, hence on all of $A$. This shows that $\tr_A$ is indeed the faithful canonical tracial state on $A$.
\end{proof}

Based on the above proposition, we make the following convention for the rest of the paper:

\begin{convention}\label{con:measures_traces}
We fix a second countable, locally compact abelian group $G$, and a closed, cocompact subgroup $\Sub$ of $\tfp{G}$. We fix Haar measures on $G$ and $\Sub$. If $\Sub$ is a lattice in $\tfp{G}$, we assume the counting measure on $\Sub$. From these measures, we obtain measures on $\widehat{G}$, $\tfp{G}$ and $\Sub^{\circ}$ as in \Cref{Section:Gabor-LCA}. Note that the measure on $\Sub^{\circ}$ will be the counting measure scaled by a factor of $s(\Sub)^{-1}$. Let $A = C^*(\Sub,c)$ and $B = C^*(\Sub^{\circ},\overline{c})$, so that the Heisenberg module $\heis{G}{\Sub}$ is an imprimitivity $A$-$B$-bimodule. We assume the canonical faithful trace  $\tr_B$ on $B$ given by $\tr_B(b) = b(0)$ for $b \in \ell^1(\Sub^{\circ},c)$. We equip $A$ with the possibly unbounded trace $\tr_A$ induced from $\tr_B$ as in \cref{prop:induced_trace_heisenberg_module}. In particular, if $\Sub$ is a lattice, then $\tr_A$ is the canonical faithful tracial state on $A$.
\end{convention}

In the following proposition, we compute the localization of the Heisenberg module associated to a cocompact subgroup $\Sub \subseteq \tfp{G}$.

\begin{proposition}\label{prop:loc_heis}
Let $G$ denote a second countable locally compact abelian group, and let $\Sub$ be a closed, cocompact subgroup of $\tfp{G}$. Then the localization $H$ of the Heisenberg module $\heis{G}{\Sub}$ with respect to either of the traces on $C^*(\Sub,c)$ and $C^*(\Sub^{\circ},\overline{c})$ can be identified with $L^2(G)$ in such a way that the diagram of inclusions commutes: 
\begin{center}
	\begin{tikzpicture}[scale=1.0]
	\node (1) at (0,2) {$S_0 (G)$};
	\node (2) at (3,2) {$\heis{G}{\Sub}$};
	\node (3) at (0,0) {$L^2 (G)$};
	\node (4) at (3,0) {$H$};
	\path[right hook->,font=\scriptsize]
	(1) edge[auto] node[auto] {} (2)
	(1) edge node[auto,swap] {} (3)
	(2) edge node[auto] {} (4);
	\path[->,font=\scriptsize]
	(3) edge node[auto] {$\cong$} (4);
	\end{tikzpicture}
\end{center}
Thus, the Heisenberg module can be continuously embedded into $L^2(G)$, with
\begin{equation}
    \| \eta \|_{2} \leq s(\Sub)^{1/2} \| \eta \|_{\heis{G}{\Sub}} \label{eq:heis_norm_dominance}
\end{equation}
for all $\eta \in \heis{G}{\Sub}$. In particular, if $(\eta_n)_n$ is a sequence in $\heis{G}{\Sub}$ that converges to an element $\eta \in \heis{G}{\Sub}$ in the $\heis{G}{\Sub}$-norm, then $(\eta_n)_n$ also converges to $\eta$ in the $L^2(G)$-norm.
\end{proposition}

\begin{proof}
Let $\xi, \eta \in S_0(G)$. Then $\lhs{\xi}{\eta} \in S_0(\Sub,c)$ by \cref{prop:heisenberg_module}, and so by \eqref{eq:mod_inner_prod} and \cref{prop:induced_trace_heisenberg_module} we obtain
\[ \langle \xi, \eta \rangle_{\tr_A} = \tr_A( \lhs{\xi}{\eta} ) = \lhs{\xi}{\eta}(0) = \langle \xi, \pi(0) \eta \rangle = \langle \xi, \eta \rangle. \]
This shows that $\langle \cdot, \cdot \rangle_{\tr}$ and $\langle \cdot, \cdot \rangle$ agree on the dense subspace $S_0(G)$ of $H$. Hence, the localization $H$ can be identified with $L^2(G)$ in such a way that the above diagram commutes. Moreover, since $\Vert \tr_B\Vert = \tr_B (1_B) = s(\Sub)$, see \autoref{rmk:not_state}, we have
\begin{equation*}
    \Vert \eta \Vert_{2}^2 = \langle \eta , \eta \rangle = \tr_B (\rhs{\eta}{\eta}) \leq \Vert \tr_B \Vert \Vert \rhs{\eta}{\eta}\Vert = s(\Sub) \Vert \eta \Vert_{\E}^2.
\end{equation*}
This implies \eqref{eq:heis_norm_dominance}.
\end{proof}

\cref{prop:loc_heis} embeds the Heisenberg module as a dense subspace of $L^2(G)$, and allows us to think of $\heis{G}{\Sub}$ as a function space.

\subsection{Applications to Gabor analysis}

In light of \cref{prop:loc_heis} and \cref{prop:bimodule_localization}, it follows that every adjointable $C^* (\Sub, c)$-module operator $\heis{G}{\Sub} \to \heis{G}{\Sub}$ has a unique extension to a bounded linear map $L^2(G) \to L^2(G)$. The following lemma states that when $\eta,\gamma \in S_0(G)$, the extension of the adjointable operator $\Theta_{\eta,\gamma}$ on $\heis{G}{\Sub}$ to a bounded linear operator on $L^2(G)$ is equal to $S_{\eta,\gamma}$. This will be generalized to functions $\eta,\gamma \in \heis{G}{\Sub}$ in \cref{thm:loc_operators}. The lemma was observed in \cite{Lu09} in the case of $G = \R^d$, but without using the language of localization. It was also covered in greater generality in \cite[Theorem 3.14]{JaLu18}.
\begin{lemma}\label{lem:feichtinger_op_loc}
Let $\eta,\gamma \in S_0(G)$. Then the module frame-like operator $\modft_{\eta,\gamma} \colon \heis{G}{\Sub} \to \heis{G}{\Sub}$ extends to the Gabor frame-like operator $S_{\eta,\gamma} \colon L^2(G) \to L^2(G)$.
\end{lemma}

\begin{proof}
Suppose $\eta,\gamma \in S_0(G)$. To begin with, let $\xi \in S_0(G)$. Then by \cref{prop:heisenberg_module}, $\lhs{\xi}{\eta} \in S_0(\Sub,c)$, and consequently $\lhs{\xi}{\eta} \cdot \gamma \in S_0(G)$. Moreover, equations \eqref{eq:mod_inner_prod} and \eqref{eq:mod_action} give that
\[ \modft_{\eta,\gamma} \xi = \lhs{\xi}{\eta} \gamma = \int_{\Sub} \lhs{\xi}{\eta}(z) \pi(z) \gamma \, \dif z = \int_{\Sub} \langle \xi, \pi(z) \eta \rangle \pi(z) \gamma \, \dif z = \ft_{\eta,\gamma} \xi .\]
Now let $\xi \in \heis{G}{\Sub}$, and suppose $(\xi_n)_n$ is a sequence in $S_0(G)$ that converges to $\xi$ in the $\heis{G}{\Sub}$-norm. Then by continuity, $\modft_{\eta,\gamma} \xi = \lim_n \modft_{\eta,\gamma} \xi_n$ in the $\heis{G}{\Sub}$-norm. By 
\autoref{prop:loc_heis}, the sequence $(\xi_n)_n$ also converges to $\xi$ in the $L^2(G)$-norm, and so by continuity, $S_{\eta,\gamma} \xi = \lim_n S_{\eta,\gamma} \xi_n$ in the $L^2(G)$-norm. From what we already proved for functions in $S_0(G)$, we obtain that $\modft_{\eta,\gamma} \xi = \ft_{\eta,\gamma} \xi$ (as elements of $L^2(G)$).

But this shows that $\ft_{\eta,\gamma} |_{\heis{G}{\Sub}} = \modft_{\eta,\gamma}$, and since the extension of $\ft_{\eta,\gamma} |_{\heis{G}{\Sub}}$ to $L^2(G)$ is $\ft_{\eta,\gamma}$, we conclude that the extension of $\modft_{\eta,\gamma}$ to $L^2(G)$ is $\ft_{\eta,\gamma}$.
\end{proof}
The following lemma was also noted in \cite[Lemma 3.6]{JaLu18}. We give a different proof here which uses localization.
\begin{lemma}\label{lem:norm_expressions}
Let $\eta \in S_0(G)$. Then the Heisenberg module norm of $\eta$ can be expressed in the following ways:
\begin{align}
    \| \eta \|_{\heis{G}{\Sub}} &= \| \an_{\eta} \| \label{eq:norm_expressions_1} \\
    &= \| \ft_{\eta} \label{eq:norm_expressions_2} \|^{1/2} \\
    &= \sup_{\| \xi \|_2 = 1} \left( \int_{\Sub} | \langle \xi, \pi(z) \eta \rangle |^2 \, \dif z \right)^{1/2} \label{eq:norm_expressions_3} \\
    &= \inf \{ D^{1/2} : \text{ $D$ is a Bessel bound for $\mathcal{G}(\eta;\Sub)$ } \label{eq:norm_expressions_4} \} .
\end{align}
\end{lemma}

\begin{proof}
By \cref{prop:rank_one_operator_norm}, the Heisenberg module norm of $\eta$ is given by $\| \eta \|_{\heis{G}{\Sub}} = \| \modft_{\eta} \|_{\Adj(\heis{G}{\Sub})}^{1/2}$. Since $\eta \in S_0(G)$, we get from \cref{lem:feichtinger_op_loc} and \eqref{eq:operator_norm_equality} that
\[ \| \eta \|_{\heis{G}{\Sub}} = \| \ft_{\eta} \|_{\Adj(L^2(G))}^{1/2} .\]
Now from the equality $\ft_{\eta} = \an_{\eta}^* \an_{\eta}$ 
it follows that $\| \ft_{\eta} \|^{1/2} = \| \an_{\eta} \|$. This takes care of \eqref{eq:norm_expressions_1} and \eqref{eq:norm_expressions_2}. The expressions in \eqref{eq:norm_expressions_3} and \eqref{eq:norm_expressions_4} are well-known for the operator norm $\| \an_{\eta} \|$.
\end{proof}


We are now ready to prove the first of our main results:

\begin{sloppypar}
\begin{theorem}\label{thm:bessel}
Let $G$ be a second countable, locally compact abelian group, and let $\Sub$ be a closed, cocompact subgroup of $\tfp{G}$. If $\eta \in \heis{G}{\Sub}$, then $\mathcal{G}(\eta;\Sub)$ is a Bessel family for $L^2(G)$. That is, there exists a $D > 0$ such that
\[ \int_{\Sub} | \langle \xi, \pi(z) \eta \rangle |^2 \, \dif z \leq D \| \xi \|_2^2 \]
for all $\xi \in L^2(G)$. Consequently, the analysis, synthesis and frame-like operators ${ \an_{\eta}\colon L^2(G) \to L^2 (\Sub) }$, $\sy_{\eta}\colon L^2 (\Sub)\to L^2 (G)$, $\ft_{\eta,\gamma}\colon L^2 (G) \to L^2 (G)$ are all well-defined, bounded linear operators for $\eta,\gamma \in \heis{G}{\Sub}$.
\end{theorem}
\end{sloppypar}

\begin{proof}
Let $\eta \in \heis{G}{\Sub}$, and let $(\eta_n)_n$ be a sequence in $S_0(G)$ with
\[ \lim_{n \to \infty} \Vert \eta - \eta_n\Vert_{\heis{G}{\Sub}} = 0. \]
Since $\eta_n \in S_0 (G)$ for all $n$, $\mathcal{G}(\eta,\Sub)$ is a Bessel family for all $n$ by \cref{prop:feichtinger_algebra}. Denote by $D_n$ the optimal Bessel bound of $\mathcal{G}(\eta;\Sub)$ for each $n$, which by \eqref{eq:norm_expressions_4} in \cref{lem:norm_expressions} is equal to $\| \eta_n \|_{\heis{G}{\Sub}}^2$. Since $( \eta_n)_n$ is convergent in the Heisenberg module norm, it follows that $( \| \eta_n \|_{\heis{G}{\Sub}} )_{n=1}^\infty$ is bounded, and so $(D_n)_{n=1}^\infty$ is bounded, by $D$ say. We then have that
\begin{equation*}
    \int_{\Sub} | \hs{\xi}{\pi (z) \eta_n}|^2 \dif z \leq D_n \Vert \xi \Vert^2_2 \leq D \Vert \xi \Vert_2^2
\end{equation*}
for every $\xi \in L^2 (G)$ and every $n \in \N$. Since $(\eta_n)_n \to \eta$ in $\heis{G}{\Sub}$, we have from \cref{prop:loc_heis} that $(\eta_n)_n \to \eta$ in $L^2(G)$ as well. Hence, continuity of the inner product gives for each $z \in \Sub$ and each $\xi \in L^2(G)$ that
\begin{equation*}
    \lim_{n\to \infty} | \hs{\xi}{\pi (z) \eta_n}|^2  = | \hs{\xi}{\pi(z) \eta}|^2.
\end{equation*}
By Fatou's lemma, we obtain for every $\xi \in L^2(G)$ that
\begin{equation*}
    \int_{\Sub} | \hs{\xi}{\pi (z) \eta}|^2 dz \leq \liminf_{n \to \infty} \int_{\Sub} | \hs{\xi}{\pi(z)\eta_n}|^2 dz \leq D\Vert \xi \Vert^2.
\end{equation*}
This proves that $\mathcal{G}(\eta;\Sub)$ is a Bessel family.
\end{proof}

We are now able to extend the description of the Heisenberg module norm given in \cref{lem:norm_expressions} for functions in $S_0(G)$ to all of $\heis{G}{\Sub}$.

\begin{proposition}\label{prop:norm_sequel}
Let $\eta \in \heis{G}{\Sub}$. Then the module norm of $\eta$ can be expressed in the following ways:
    \begin{align}
        \Vert \eta \Vert_{\heis{G}{\Sub}} &= \Vert \an_{\eta}\Vert \label{eq:norm_sequel1} \\
        &= \Vert \ft_{\eta} \Vert^{1/2} \label{eq:norm_sequel2} \\
        &= \sup_{\Vert \xi \Vert =1} \bigg ( \int_{\Sub} | \hs{\xi}{\pi(z) \eta}|^2 dz \bigg)^{1/2} \label{eq:norm_sequel3} \\
        &= \inf \{ D^{1/2} : \text{$D$ is a Bessel bound for $\mathcal{G}(\eta;\Sub)$} \}. \label{eq:norm_sequel4}
    \end{align}
\end{proposition}

\begin{proof}
Let $\eta \in \heis{G}{\Sub}$. We will show that $\| \eta \|_{\heis{G}{\Sub}} = \| C_{\eta} \|$. Once this is shown, the rest of the expressions for $\| \eta \|_{\heis{G}{\Sub}}$ follow just as in the proof of \cref{lem:norm_expressions}.

Let $(\eta_n)_{n=1}^\infty$ be a sequence in $S_0(G)$ such that
\[ \lim_{n \to \infty} \| \eta - \eta_n \|_{\heis{G}{\Sub}} = 0 .\]
Then $(\eta_n)_n$ is a Cauchy sequence in the Heisenberg module norm, and so for every $\epsilon > 0$ there exists $N \in \N$ such that for all $m,n \geq N$ we have that
\[ \| \eta_m - \eta_n \|_{\heis{G}{\Sub}} < \epsilon. \]
Since $\eta_n \in S_0(G)$ for all $n \in \N$ and $S_0(G)$ is a subspace of $L^2(G)$, we have that $\eta_m - \eta_n \in S_0(G)$ for all $m,n \in \N$, and so by 
\cref{lem:norm_expressions}, we can write
\[ \| \eta_m - \eta_n \|_{\heis{G}{\Sub}} = \| C_{\eta_m - \eta_n} \| = \| C_{\eta_m} - C_{\eta_n} \| .\]
But then by the above, we obtain that the sequence of operators $(C_{\eta_n})_{n=1}^\infty$ is Cauchy in $\Adj(L^2(G), L^2 (\Sub))$, and so by completeness, there exists $T \in \Adj(L^2(G), L^2 (\Sub))$ such that
\[ \lim_{n \to \infty}\| T - C_{\eta_n} \| =0 .\]
Now fix $\xi \in L^2(G)$. Then we have that
\[ \lim_{n \to \infty}\| T\xi - C_{\eta_n} \xi \|_2 = 0 .\]
It is well-known that this implies the existence of a subsequence $(C_{\eta_{n_k}} \xi)_{k=1}^\infty$ that converges pointwise almost everywhere to $T \xi$ (see for instance \cite[Theorem 3.12]{Ru87}). However, since $(\eta_n)_n$ converges to $\eta$ in the $L^2(G)$-norm by 
\autoref{prop:loc_heis}, we have that
\[ \lim_{n \to \infty} C_{\eta_n}\xi (z) = \lim_{n \to \infty} \langle \xi, \pi(z) \eta_n \rangle = \langle \xi, \pi(z) \eta \rangle = C_{\eta} \xi (z) \] 
for every $z \in \Sub$. Hence $(C_{\eta_n} \xi)_n$ converges pointwise to $C_{\eta} \xi$, and it follows that $(C_{\eta_{n_k}} \xi)_k$ converges pointwise to $C_{\eta} \xi$ as well. This shows that $C_{\eta} \xi = T \xi$ almost everywhere, and so they represent the same element in $L^2(\Sub)$. Since $\xi$ was arbitrary, it follows that $C_{\eta} = T$, and so we have that
\[ \lim_{n \to \infty} \| C_{\eta} - C_{\eta_n} \| = 0 .\]
This implies that
\[ \| \eta \|_{\heis{G}{\Sub}} = \lim_{n \to \infty} \| \eta_n \|_{\heis{G}{\Sub}} = \lim_{n \to \infty} \| C_{\eta_n} \| = \| C_{\eta} \| .\]

\end{proof}

Let $\bes{G}{\Sub}$ denote the set of $\xi \in L^2(G)$ such that $\mathcal{G}(\eta;\Sub)$ is a Bessel family for $L^2(G)$. Then $\bes{G}{\Sub}$ is a Banach space when equipped with the norm 
\begin{equation}
    \Vert \eta \Vert_{\bes{G}{\Sub}} = \Vert C_{\eta} \Vert = \inf \{ D^{1/2} : \text{ $D$ is a Bessel bound for $\mathcal{G}(\eta;\Sub)$ } \}.
\end{equation}
By \cref{prop:heisenberg_module}, the Heisenberg module $\heis{G}{\Sub}$ is the completion of $S_0(G)$ with respect to the Heisenberg module norm. But by using our embedding of $\heis{G}{\Sub}$ into $L^2(G)$ in \cref{prop:loc_heis} and the expression of the Heisenberg module norm provided in \cref{prop:norm_sequel}, we obtain a concrete description of $\heis{G}{\Sub}$ as a subspace of $L^2(G)$. In the following proposition, we use the notation from \eqref{eq:inner_prod_rep}.

\begin{proposition}
\label{Prop:heis-completion-bessel}
Let $G$ be a second countable, locally compact abelian group, and let $\Sub$ be a closed, cocompact subgroup of $\tfp{G}$. Then the Heisenberg module $\heis{G}{\Sub}$ is the completion of $S_0(G)$ in $\bes{G}{\Sub}$. The bimodule structure can be described as follows: Let $a \in C^*(\Sub,c)$, $b \in C^*(\Sub^{\circ},\overline{c})$ and $\xi \in \heis{G}{\Sub}$. Then
\begin{align}
    a \cdot \xi &= \pi_{\Sub}(a) \xi \label{eq:module_action_concrete} \\
    \xi \cdot b &= \pi_{\Sub^{\circ}}^*(b) \xi \label{eq:module_right_action_concrete}
\end{align}
\end{proposition}

\begin{proof}
By \cref{prop:heisenberg_module}, we know that $\heis{G}{\Sub}$ is the completion of $S_0(G)$ with respect to the Heisenberg module norm. By \cref{prop:loc_heis}, we know that $\heis{G}{\Sub}$ is continuously embedded into $L^2(G)$ in a way that respects the embedding of $S_0(G)$ into $L^2(G)$. By \cref{prop:norm_sequel}, we have a description of the Heisenberg module norm as $\| \eta \|_{\heis{G}{\Sub}} = \| \eta \|_{\bes{G}{\Sub}}$. It follows that $\heis{G}{\Sub}$ is the completion of $S_0(G)$ with respect to the norm of $\bes{G}{\Sub}$.

To see that \eqref{eq:module_action_concrete} holds, let $a \in C^*(\Sub,c)$ and $\xi \in \heis{G}{\Sub}$. Let $(a_n)_n$ be a sequence in $S_0(\Sub,c)$ such that $\lim_{n \to \infty} a_n = a$ in $C^*(\Sub,c)$. Let $(\xi_n)_n$ be a sequence in $S_0(G)$ such that $\lim_{n \to \infty} \xi_n = \xi$ in $\heis{G}{\Sub}$. Then by continuity of the left action of $C^*(\Sub,c)$ on $\heis{G}{\Sub}$, we have that
\[ a \cdot \xi = \left(\lim_n a_n \right) \cdot \left(\lim_n \xi_n \right) = \lim_n \left(a_n \cdot \xi_n \right) = \lim_n \pi(a_n) \xi_n \]
in $\heis{G}{\Sub}$. The last equality follows from the description of $a \cdot \xi$ for $a \in S_0(\Sub,c)$ and $\xi \in S_0(G)$ as $\pi(a) \xi$ (see \cref{prop:heisenberg_module}). Since $\xi_n \to \xi$ in the Heisenberg module norm, we have that $\xi_n \to \xi$ in the $L^2(G)$-norm. Also, since $\pi(a_n) \to \pi(a)$ in the operator norm, we have that $\pi(a_n) \xi_n \to \pi(a) \xi$ in the $L^2(G)$-norm. Hence, interchanging the $\heis{G}{\Sub}$-limit in the equation above with an $L^2(G)$-limit, we obtain that $a \cdot \xi = \pi_{\Sub}(a) \xi$.

The argument for \eqref{eq:module_right_action_concrete} is similar, as for $b \in S_0(\Sub^{\circ},\overline{c})$ and $\xi \in S_0(G)$, the definition of $\xi \cdot b$ in \cref{prop:heisenberg_module} is equal to $\pi_{\Sub^{\circ}}^*(b)\xi$. A similar approximation argument to the one above shows that $\xi \cdot b = \pi_{\Sub^{\circ}}^*(b)\xi$ also holds for $b \in C^*(\Sub^{\circ},\overline{c})$ and $\xi \in \heis{G}{\Sub}$.
\end{proof}

\begin{example}
If one sets $\Sub = \tfp{G}$, the Heisenberg module $\heis{G}{\Sub}$ is all of $L^2(G)$. To see this, note that $\Sub^{\circ} = \{0\}$. Thus, we have the identification $C^* (\Sub^{\circ}, \overline{c}) \cong \C$, where a sequence $a \in C^*(\Sub^{\circ},\overline{c}) = \C \Sub^{\circ}$ is identified with its value $a(0)$ at $0$. In this situation, the Heisenberg module $\heis{G}{\Sub}$ is a $C^*(\Sub,c)$-$\C$-imprimitivity bimodule. But then $\heis{G}{\Delta}$ is a right Hilbert $C^*$-module over $\C$, so it must be a Hilbert space (with linearity in the second argument of the inner product). The right action is given by
\[ \xi \cdot b = \sum_{w \in \Sub^{\circ}} b(w) \pi(w)^* \xi = b(0) \xi, \]
which under the identification $C^*(\Sub^{\circ},\overline{c}) \cong \C$ becomes $\xi \cdot \lambda = \xi \lambda$ for $\xi \in L^2(G)$ and $\lambda \in \C$, i.e.\ ordinary scalar multiplication. Furthermore, the inner product at the value $0$ is given by $\rhs{\xi}{\eta}(0) = \langle \pi(0) \eta, \xi \rangle = \langle \eta, \xi \rangle$, i.e.\ the right inner product is just the conjugate of the ordinary $L^2(G)$-inner product.

It follows immediately from \cref{prop:rank_one_operator_norm} that the Heisenberg module norm in this case is just the $L^2(G)$-norm, and so $\heis{G}{\Sub} = \bes{G}{\Sub} = L^2 (G)$. The statement $\bes{G}{\Sub} = L^2 (G)$ when $\Sub$ is the whole time-frequency plane is well known. Indeed, in this case the analysis operator $\an_{\eta}$ is the short-time Fourier transform. This is a bounded operator $L^2 (G)\to L^2 (\tfp{G})$ for all $\eta \in L^2 (G)$, and is invertible for any $\eta \neq 0$, see \cite[Theorem 6.2.1]{gr98}. 
\end{example}

\begin{example}
Suppose $G$ is a discrete group, and that $\Sub$ is a cocompact subgroup of $\tfp{G}$ (which must then be a lattice). Then $S_0(G) = \ell^1(G)$ (\cref{prop:feichtinger_algebra}, \ref{prop:feichtinger_algebra_it2}), and so the Heisenberg module satisfies $\ell^1(G) \subseteq \heis{G}{\Sub} \subseteq \ell^2(G)$. In particular, if $G$ is finite, then $\heis{G}{\Sub} = \ell^1(G) = \ell^2(G) =  \C G \cong \C^{|G|}$.
\end{example}

The following theorem extends \cref{lem:feichtinger_op_loc} and is one of our main results.

\begin{theorem}\label{thm:loc_operators}
Let $G$ be a second countable, locally compact abelian group, and let $\Sub$ be a closed, cocompact subgroup of $\tfp{G}$. Let $\eta,\gamma \in \heis{G}{\Sub}$. Then the module frame-like operator $\modft_{\eta,\gamma} \colon \heis{G}{\Sub} \to \heis{G}{\Sub}$ extends via localization to the Gabor frame-like operator $\ft_{\eta,\gamma} \colon L^2(G) \to L^2(G)$.
\end{theorem}

\begin{proof}
Let $(\eta_n)_{n=1}^\infty$ and $(\gamma_n)_{n=1}^\infty$ be sequences in $S_0(G)$ that converge towards $\eta$ and $\gamma$ respectively in the Heisenberg module norm. Let $\xi \in \heis{G}{\Sub}$. Then $(\modft_{\eta_n,\gamma_n} \xi)_n$ converges towards $\modft_{\eta,\gamma} \xi$ in the Heisenberg module norm. By \cref{lem:feichtinger_op_loc}, we have that $\modft_{\eta_n,\gamma_n} \xi = \ft_{\eta_n,\gamma_n} \xi$ for each $n$, and since convergence in the Heisenberg module norm implies convergence in the $L^2(G)$-norm, we have that
\begin{equation}
    \lim_{n \to \infty} \| S_{\eta_n,\gamma_n} \xi_n - \modft_{\eta,\gamma} \xi \|_2 = 0 . \label{eq:l2conv}
\end{equation}
By \cref{prop:norm_sequel} and the identity $C_{\eta - \eta_n} = C_{\eta} - C_{\eta_n}$, the sequences of operators $(C_{\eta_n})_n$ and $(C_{\gamma_n}^*)_n$ converge in the operator norm to $C_\eta$ and $C_{\gamma}^*$ respectively, and so $(S_{\eta_n,\gamma_n} )_n$ converges in the operator norm towards $S_{\eta,\gamma}$. It follows that the sequence $(S_{\eta_n,\gamma_n} \xi)_n$ converges to $S_{\eta,\gamma} \xi$ in the $L^2$-norm. But then by \eqref{eq:l2conv}, we have that $\modft_{\eta,\gamma} \xi = \ft_{\eta,\gamma} \xi$. This shows that the restriction of $\ft_{\eta,\gamma}$ to $\heis{G}{\Sub}$ is equal to $\modft_{\eta,\gamma}$, and so the unique extension of $\modft_{\eta,\gamma}$ to a bounded linear operator on $L^2(G)$ must be $\ft_{\eta,\gamma}$.
\end{proof}

We now arrive at another one of our main results. The following result was previously only known for generators in $S_0 (G)$ \cite{Lu09,JaLe16}. It states that finite module frames for $\heis{G}{\Sub}$ are exactly the generators of multi-window Gabor frames for $L^2 (G)$, where the generators are allowed to come from $\heis{G}{\Sub}$. This gives a complete description of generators of Heisenberg modules in terms of multi-window Gabor frames.
\begin{theorem}\label{thm:generators_multiwindow}
Let $G$ be a second countable, locally compact abelian group, let $\Sub$ be a closed, cocompact subgroup of $\tfp{G}$, and let $\eta_1, \ldots, \eta_k$ be elements of the Heisenberg module $\heis{G}{\Sub}$. Then the following are equivalent:
\begin{enumerate}
    \item The set $\{ \eta_1, \ldots, \eta_k \}$ generates $\heis{G}{\Sub}$ as a left $C^*(\Sub,c)$-module. That is, for all $\xi \in \heis{G}{\Sub}$ there exist $a_1, \ldots, a_k \in C^*(\Sub,c)$ such that
    \[ \xi = \sum_{j=1}^k a_j \cdot \eta_j .\]
    \item The system
    \[ \mathcal{G}(\eta_1, \ldots, \eta_k ; \Sub) = \{ \pi(z) \eta_j : z \in \Sub, 1 \leq j \leq k \} \]
    is a multi-window Gabor frame for $L^2(G)$.
\end{enumerate}
\end{theorem}

\begin{proof}
By \cref{prop:frame_finite}, the set $\{ \eta_1, \ldots, \eta_k \}$ is a generating set for $\heis{G}{\Sub}$ if and only if the sequence $(\eta_1, \ldots, \eta_k)$ is a module frame for $\heis{G}{\Sub}$. By \cref{prop:frame_invertible}, this happens if and only if $\modft = \modft_{(\eta_j)_{j=1}^k} = \sum_{j=1}^k \modft_{\eta_j}$ is invertible as an element of $\Adj_{C^* (\Sub, c)}(\heis{G}{\Sub})$. By \cref{thm:loc_operators} and linearity of the localization map $\Adj_{C^* (\Sub , c)}(\heis{G}{\Sub}) \hookrightarrow \Adj(L^2(G))$, this operator extends via localization to the Gabor multi-window frame operator $S = \ft_{(\eta_j)_{j=1}^k} = \sum_{j=1}^k \ft_{\eta_j}$ on $L^2(G)$. Since the localization map $\Adj(\heis{G}{\Sub}) \hookrightarrow \Adj(L^2(G))$ is an inclusion of unital $C^*$-subalgebras, it follows by inverse closedness \cite[Theorem 2.1.11]{Mu90} that $\modft$ is invertible in $\Adj(\heis{G}{\Sub})$ if and only if $\ft$ is invertible in $\Adj(L^2(G))$. But by \cref{prop:Gabor-ft-inv-iff-Gabor-system}, the latter happens if and only if $\mathcal{G}(\eta_1, \ldots, \eta_k ; \Sub)$ is a frame for $L^2(G)$.
\end{proof}

\subsection{The fundamental identity of Gabor analysis}
So far we have considered a closed subgroup $\Sub$ of $\tfp{G}$, and from this we built the Heisenberg module $\heis{G}{\Sub}$, which is a $C^* (\Sub,c)$-$C^*(\Sub^{\circ}, \overline{c})$-imprimitivity bimodule. We focused specifically on the case when $\Sub$ is cocompact, since this implies $\Sub^{\circ}$ is discrete and hence $C^*(\Sub^{\circ}, \overline{c})$ is unital. 
By \cite[p. 5]{JaLe16}, $\Sub^{\circ}$ is identical to the annihilator $\Sub^{\perp}$ (also defined in the same article) up to a measure-preserving change of coordinates, and it is also the case that $(\Sub^{\perp})^{\perp} = \Sub$ by \cite[Proposition 3.6.1]{DeEc14}. Hence $(\Sub^{\circ})^{\circ} = \Sub$. Imposing the restriction that $\Sub^{\circ}$ also be cocompact, which implies that both $\Sub$ and $\Sub^{\circ}$ are lattices, we could build $\heis{G}{\Sub^{\circ}}$ and ask how it relates to $\heis{G}{\Sub}$. The following proposition shows that the relationship is just about as good as we could hope for.
\begin{proposition}\label{prop:Equal-heis-mods}
Let $\Sub$ be a lattice in $\tfp{G}$. Then $\heis{G}{\Sub} = \heis{G}{\Sub^{\circ}}$ as subspaces of $L^2(G)$, and $\Vert \eta \Vert_{\heis{G}{\Sub^{\circ}}} = s(\Sub)^{-1/2}\Vert \eta \Vert_{\heis{G}{\Sub}}$ for all $\eta \in \heis{G}{\Sub}$.
\end{proposition}
\begin{proof}
Note first that since $\Sub$ is a lattice, so is $\Sub^{\circ}$. In particular, $\Sub^{\circ}$ is a cocompact subgroup, so all the results in this section for $\Sub$ apply just as well for $\Sub^{\circ}$. Hence, by \cref{prop:heisenberg_module} and \cref{Prop:heis-completion-bessel}, one obtains the Heisenberg module $\heis{G}{\Sub^{\circ}}$ as a completion of $S_0(G)$ in $\bes{G}{\Sub^{\circ}}$, which is a $C^*(\Sub^{\circ}, c)$-$C^*(\Sub,\overline{c})$-imprimitivity bimodule. Note that in the construction of $\heis{G}{\Sub^{\circ}}$ we put on $\Sub^{\circ}$ the counting measure, and on $\Sub$ the counting measure scaled with $s(\Sub^{\circ})^{-1}$ as per \cref{con:measures_traces}. Denote the left inner product on $\heis{G}{\Sub^{\circ}}$ by $\lhs{\cdot}{\cdot}'$ and the right inner product by $\rhs{\cdot}{\cdot}'$. 

Let $\eta \in S_0 (G)$. Denote by $\pi_{\Sub^{\circ}}^*$ the $C^*$-algebra representation as in the discussion following \cref{prop:heisenberg_module}. Denote by $\pi_{\Sub^{\circ}}$ the representation $\pi_{\Sub}$ in the same discussion, but with $\Sub$ replaced with $\Sub^{\circ}$. In other words, $\pi_{\Sub^{\circ}}^*$ is a representation of $C^*(\Sub^{\circ}, \overline{c})$ on $L^2(G)$, while $\pi_{\Sub^{\circ}}$ is a representation of $C^*(\Sub^{\circ}, c)$ on $L^2(G)$. Keeping in mind the right measures, we have that
\begin{align*}
    \pi_{\Sub^{\circ}}^*(\rhs{\eta}{\eta}) &= s(\Sub)^{-1} \sum_{w \in \Sub^{\circ}} \langle \pi(w) \eta, \eta \rangle \pi(w)^* \\
    \pi_{\Sub^{\circ}}(\lhs{\eta}{\eta}') &= \sum_{w \in \Sub^{\circ}} \langle \eta, \pi(w) \eta \rangle \pi(w).
\end{align*}
From the above we see that
\begin{equation}
    (\pi_{\Sub^{\circ}}^*(\rhs{\eta}{\eta}))^* = s(\Sub)^{-1} \pi_{\Sub^{\circ}}(\lhs{\eta}{\eta}') . \label{eq:operators_rel}
\end{equation}
Using the faithfulness of the integrated representations (\cref{prop:faithful_rep}), we obtain for all $\eta \in S_0(G)$ that
{\allowdisplaybreaks \begin{align*}
    \| \eta \|_{\heis{G}{\Sub^{\circ}}}^2 &= \| \lhs{\eta}{\eta}' \|_{C^*(\Sub^{\circ},c)} \\
    &= \| \pi_{\Sub^{\circ}}( \lhs{\eta}{\eta}') \| && \text{by faithfulness of $\pi_{\Sub^{\circ}}$} \\
    &= \| \pi_{\Sub^{\circ}}( \lhs{\eta}{\eta}')^* \| \\
    &= \| s(\Sub)^{-1} \pi_{\Sub^{\circ}}^*(\rhs{\eta}{\eta}) \| && \text{by \eqref{eq:operators_rel} } \\
    &= s(\Sub)^{-1} \| \rhs{\eta}{\eta} \|_{C^*(\Sub^{\circ},\overline{c})} && \text{by faithfulness of $\pi_{\Sub^{\circ}}^*$ } \\
    &= s(\Sub)^{-1} \| \lhs{\eta}{\eta} \|_{C^*(\Sub,c)} && \text{by \cref{prop:rank_one_operator_norm}} \\
    &= s(\Sub)^{-1} \| \eta \|_{\heis{G}{\Sub}}^2.
\end{align*}}
By the above, we have that a sequence $(\eta_n)_n$ in $S_0(G)$ is Cauchy in the $\heis{G}{\Sub}$-norm if and only if it is Cauchy in the $\heis{G}{\Sub^{\circ}}$-norm. Thus, the sequence has a limit in $\heis{G}{\Sub}$-norm if and only if it has a limit in $\heis{G}{\Sub^{\circ}}$-norm. Since both of these norms dominate the $L^2(G)$-norm, it follows if $(\eta_n)_n$ is Cauchy in either of the norms, the limit in either of the norms give the same element in $L^2(G)$. It follows that $\heis{G}{\Sub^{\circ}} = \heis{G}{\Sub}$, with $\| \eta \|_{\heis{G}{\Sub^{\circ}}} = s(\Sub)^{1/2} \| \eta \|_{\heis{G}{\Sub}}$ for all $\eta \in \heis{G}{\Sub}$.
\end{proof}

Finally, we show that the fundamental identity of Gabor analysis (or FIGA) \cite[Theorem 4.5]{FeLu06} holds when all involved functions are in $\heis{G}{\Sub}$ when $\Sub$ is a lattice.
In the following we let $\ft^{\Sub}_{\eta, \gamma}$ be the operator of \eqref{Eq:frame-like-operator}, and let $\ft^{\Sub^{\circ}}_{\eta , \gamma}$ denote the operator of \eqref{Eq:frame-like-operator} but with $\Sub^{\circ}$ instead of $\Sub$.
It is already known that FIGA holds for functions in $S_0 (G)$ even when $\Sub$ is just a closed subgroup of the time-frequency plane $\tfp{G}$ \cite[Corollary 6.3]{JaLe16}. However, in the proof of \cref{prop:figa} below we shall need to apply \cref{thm:loc_operators} to frame operators with respect to $\Sub$ and with respect to $\Sub^{\circ}$. This then requires that both $\Sub$ and $\Sub^\circ$ are cocompact, hence they are both lattices. 
With these restrictions, FIGA is the statement
\begin{equation}
    \sum_{z\in \Sub}  \hs{\eta}{\pi (z) \gamma}  \hs{ \pi (z) \xi }{ \psi}  = \frac{1}{s(\Sub)} \sum_{z^{\circ}\in \Sub^{\circ}} \hs{ \xi}{ \pi (z^{\circ}) \gamma} \hs { \pi(z^{\circ}) \eta}{\psi} 
\end{equation}
for $\eta,\gamma,\xi, \psi \in S_0(G)$. In short form it is just the statement
\begin{equation}
\label{eq:FIGA}
    \ft^{\Sub}_{\gamma,\xi}\eta = s(\Sub)^{-1} \ft^{\Sub^{\circ}}_{\gamma,\eta}\xi
\end{equation}
for $\eta,\gamma,\xi \in S_0(G)$. With the restriction that $\Sub$ is a lattice in $\tfp{G}$, the following proposition extends the range for the FIGA (for the particular lattice $\Sub$) to functions in $\heis{G}{\Sub}$.
\begin{proposition}\label{prop:figa}
Let $G$ be a second countable, locally compact abelian group, and let $\Sub$ be a lattice in $\tfp{G}$. Then \eqref{eq:FIGA} holds for $\eta, \gamma, \xi \in \heis{G}{\Sub}$. 
\end{proposition}
\begin{proof} 
Let $(\eta_n)_n$, $(\gamma_n)_n$ and $(\xi_n)_n$ be sequences in $S_0(G)$ that converge to $\eta$, $\gamma$ and $\xi$, respectively, in the $\heis{G}{\Sub}$-norm. By \autoref{prop:Equal-heis-mods},  the same is true in $\heis{G}{\Sub^{\circ}}$-norm. Then, since the fundamental identity of Gabor analysis applies for functions in $S_0 (G)$ by \cite[Corollary 6.3]{JaLe16},
we have 
\begin{equation*}
    S^{\Sub}_{\gamma_n,\xi_n}\eta_n = s(\Sub)^{-1} S^{\Sub^{\circ}}_{\gamma_n,\eta_n}\xi_n
\end{equation*}
for all $n \in \N$. By \cref{thm:loc_operators} we have
\begin{align*}
    \lim_{n \to \infty} \| \ft^{\Delta}_{\gamma, \xi}\eta - \ft^{\Delta}_{ \gamma_n, \xi_n}\eta_n \|_{\heis{G}{\Sub}} &= 0 \\
    \lim_{n \to \infty} \| \ft^{\Sub^{\circ}}_{\gamma,\eta}\xi - \ft^{\Sub^{\circ}}_{\gamma_n, \eta_n} \xi_n \|_{\heis{G}{\Sub^{\circ}}} &= 0.
\end{align*}
Since convergence in the Heisenberg module norm implies convergence in the $L^2(G)$-norm, we conclude that the following equality holds in $L^2(G)$, were the limits are taken in $L^2(G)$:
\begin{equation*}
    \ft^{\Sub}_{\gamma, \xi}\eta 
    = \lim_{n\to \infty} \ft^{\Sub}_{\gamma_n, \xi_n}\eta_n 
    = \lim_{n\to \infty} s(\Sub)^{-1} \ft^{\Sub^{\circ}}_{\gamma_n, \eta_n} \xi_n 
    = s(\Sub)^{-1} \ft^{\Sub^{\circ}}_{\gamma,\eta}\xi.
\end{equation*}
\end{proof}

\begin{remark}
As already mentioned, the FIGA holds for functions in $S_0(G)$ even when $\Delta$ is only a closed subgroup for $\tfp{G}$. The techniques in this paper are based on localization of $A$-$B$-imprimitivity bimodules, which requires that we have a finite trace on at least one of the algebras $A$ and $B$. Therefore the assumption that $\Delta$ is a lattice in $\tfp{G}$ is necessary for our approach to the FIGA. There might be another technique that allows for an extension of the FIGA to $\heis{G}{\Delta}$ for $\Delta$ only a closed subgroup of $\tfp{G}$ that the authors are not aware of. We remark again that for the existence of Gabor frames over a closed subgroup $\Delta$ of $\tfp{G}$, it is necessary that $\Delta$ is cocompact in $\tfp{G}$, which is the setting for most of the results in the paper.
\end{remark}

\bibliography{bibl} 
\bibliographystyle{abbrv}

\end{document}